\pdfoutput=1
\RequirePackage{ifpdf}
\ifpdf % We~are running pdfTeX in pdf mode
\documentclass[pdftex]{sigma}
\else
\documentclass{sigma}
\fi

\usepackage{makecell,tikz}
\numberwithin{equation}{section}

\newtheorem{Theorem}{Theorem}[section]
\newtheorem*{Theorem*}{Theorem}
\newtheorem{Corollary}[Theorem]{Corollary}
\newtheorem{Lemma}[Theorem]{Lemma}
\newtheorem{Proposition}[Theorem]{Proposition}
 { \theoremstyle{definition}
\newtheorem{Definition}[Theorem]{Definition}

\newtheorem{Example}[Theorem]{Example}
\newtheorem{Remark}[Theorem]{Remark} }
\numberwithin{table}{section}
\numberwithin{figure}{section}

\newcommand{\rFs}[5]{\,_{#1}F_{#2} \left( \genfrac{.}{.}{0pt}{}{#3}{#4}
	\ ;#5 \right)}

\begin{document}
\allowdisplaybreaks

\newcommand{\arXivNumber}{2210.03041}

\renewcommand{\PaperNumber}{055}

\FirstPageHeading

\ShortArticleName{Matrix Spherical Functions for $(\mathrm{SU}(n+m),\mathrm{S}(\mathrm{U}(n)\times \mathrm{U}(m)))$: Two Specific Classes}

\ArticleName{Matrix Spherical Functions\\ for $\boldsymbol{(\mathrm{SU}(n+m),\mathrm{S}(\mathrm{U}(n)\times \mathrm{U}(m)))}$:\\ Two Specific Classes}

\Author{Jie LIU}

\AuthorNameForHeading{J.~Liu}

\Address{Institute for Mathematics, Astrophysics and Particle Physics, Radboud University Nijmegen,\\ Heyendaalseweg 135, 6525 AJ Nijmegen, The Netherlands}
\Email{\href{mailto:liujiemath@hotmail.com}{liujiemath@hotmail.com}}

\ArticleDates{Received October 18, 2022, in final form July 13, 2023; Published online August 04, 2023}

\Abstract{We consider the matrix spherical function related to the compact symmetric pair $(G,K)=(\mathrm{SU}(n+m),\mathrm{S}(\mathrm{U}(n)\times\mathrm{U}(m)))$. The irreducible $K$ representations $(\pi,V)$ in the~${\rm U}(n)$ part are considered and the induced representation $\mathrm{Ind}_K^G\pi$ splits multiplicity free. In this case, the irreducible $K$ representations in the ${\rm U}(n)$ part are studied. The corresponding spherical functions can be approximated in terms of the simpler matrix-valued functions. We can determine the explicit spherical functions using the action of a differential operator. We consider several cases of irreducible $K$ representations and the orthogonality relations are also described.}


\Classification{17B10; 22E46; 33C50}

\section{Generalities}\label{subsection:notationsandpreliminaries}
\subsection{Introduction}

There is a close relation between representation theory and special functions. In this paper, we consider explicit matrix-valued polynomials, i.e., matrix spherical functions. We use the notion of spherical functions for symmetric pair in \cite{GangV,Warn-2} taking values in a matrix algebra.

The matrix-valued spherical functions of rank one type have been exploited in several cases. In \cite{KoelvPR-IMRN, KoelvPR-PRIMS}, the matrix-valued spherical functions for symmetric pair $(\mathrm{SU}(2)\times \mathrm{SU}(2),\operatorname{diag})$ arising from Koornwinder \cite{Koor-SIAM1985} are studied. In \cite{GrunPT}, the matrix-valued spherical functions for symmetric pair $(\mathrm{SU}(3),\mathrm{U}(2))$ are studied, and for the more general case for symmetric pair $(\mathrm{SU}(m+1),\mathrm{U}(m))$ they are given in \cite{PaTi, Prui-pp}. The approach in~\cite{PaTi} is to find two different differential operators and the spherical functions are the corresponding eigenfunctions. The approach in~\cite{Prui-pp} is to find the $K$ intertwiner~$j$ by Lemma~\ref{newdef}. The rank two case for symmetric pair $(\mathrm{SU}(2+m),\mathrm{S}(\mathrm{U}(2)\times \mathrm{U}(m)))$ has been studied in~\cite{KoelLiu}. The approximate spherical functions of this case can be related to the Krawtchouk polynomials, and it is also a matrix analogue of Koornwinder's~$BC_2$ orthogonal polynomials in \cite{Koor-IndagM1, Koor-IndagM2, Koor-IndagM3, Koor-IndagM4}. Moreover, in \cite{StokR} it shows the relation to mathematical physics and possible applications.

Scalar-valued spherical functions for symmetric pair $(\mathrm{SU}(n+m),\mathrm{S}(\mathrm{U}(n)\times \mathrm{U}(m)))$ are given by \cite{HeckS}. In this paper, we calculate the matrix spherical functions for the same symmetric pair. The approach to calculating the corresponding matrix spherical functions is motivated by \cite{KoelLiu}.

Now we introduce the contents of this paper. In Sections~\ref{subsec:mulfretri} and~\ref{subsec:sphfunc}, we briefly recall the definition of a multiplicity free triple and the spherical function. In Section \ref{sec:structhe}, we describe the structure theory and the representation theory in more details for the symmetric pair $(G,K)=(\mathrm{SU}(n+m),\mathrm{S}(\mathrm{U}(n)\times \mathrm{U}(m)))$. In Section \ref{sec:sphfunctoA}, we study the spherical function restricted to a subgroup $A$ of $G$, since it uniquely determines the spherical function by the Cartan decomposition. In Section \ref{sec:radialpartofthecasimiroperator}, we calculate the radial part of the Casimir operator since the spherical function is the eigenfunction of this operator. In Section \ref{geninj}, we give the simplest cases and obtain the approximate functions. It is an intermediate step for calculating the corresponding spherical functions in Sections~\ref{exa:aomega1+bomegan} and~\ref{sec:omegasbomegan}. In Section \ref{sec:matrixweight}, we study the orthogonality relations and calculate the matrix weight function.

\subsection{Multiplicity free triples}\label{subsec:mulfretri}
Let $G$ be a compact Lie group and $K$ be a compact subgroup of
$G$. We define $\pi_{\lambda}^{G}$ as the irreducible $G$ representation with the highest weight $\lambda$ and $V_{\lambda}^{G}$ as the corresponding ir\-re\-ducible~$G$ module. Also $\pi_{\mu}^{K}$ and $V_{\mu}^{K}$ can be defined similarly. We define $\big(P_{G}^{+},P_{K}^{+}\big)$ as the set of the highest weight of $(G,K)$, and $\big[\pi_{\lambda}^{G}|_{K}:\pi_{\mu}^{K}\big]$ as the multiplicity of $V_{\mu}^{K}$ in $V_{\lambda}^{G}$ decomposed as $K$ module.

Let $\mu$ be a highest weight of $K$, the triple $(G,K,\mu)$ is a multiplicity free triple if and only if $\big[\pi_{\lambda}^{G}|_{K}:\pi_{\mu}^{K}\big]\leq1$ for all $\lambda\in P_{G}^{+}$. We define
$P_{G}^{+}(\mu)=\big\{\lambda\in P_{G}^{+}\mid \big[\pi_{\lambda}^{G}|_{K}:\pi_{\mu}^{K}\big]=1\big\}$.

\subsection{Spherical functions}\label{subsec:sphfunc}

Now we give some preliminaries of spherical functions. We recall some results from \cite{GangV, Helg-1962, KoelvPR-JFA, Kram}.

Let $\mathbb{C}[G]$ be the algebra of matrix elements of finite-dimensional irreducible representations of the compact group $G=\mathrm{SU}(n+m)$. Then we have an action of $G\times G$ on $\mathbb{C}[G]$ by
\[ [(g_1,g_2)\cdot f](g)=f\big(g_1^{-1}gg_2\big),\qquad g_1,g_2,g\in G, \]
which is the biregular representation. By restriction, $\mathbb{C}[G]$ is a $K\times K$ representation. For a~fixed $K$-representation $\pi_\mu^K\colon K\rightarrow V_\mu^K$ of highest weight $\mu$, we also have $\operatorname{End}\big(V_\mu^K\big)$ as a $K\times K$ representation by
\[(k_1,k_2)\cdot T=\pi_\mu^K(k_1)T\pi_\mu^K\big(k_2^{-1}\big),\qquad k_1,k_2\in K, T\in\operatorname{End}\big(V_\mu^K\big).\]
Then we consider $\mathbb{C}[G]\otimes\operatorname{End}\big(V_\mu^K\big)$ as a space of functions $\widetilde{\Phi}\colon G\rightarrow\operatorname{End}\big(V_\mu^K\big)$, which is a $K\times K$ representation by
\[\big[(k_1,k_2)\cdot \widetilde{\Phi}\big](g)=\pi_\mu^K(k_1)\widetilde{\Phi}\big(k_1^{-1}gk_2\big)\pi_\mu^K\big(k_2^{-1}\big),\qquad k_1,k_2\in K,g\in G.\]

\begin{Definition}\label{def}
A matrix spherical function is an element of $(\mathbb{C}[G]\otimes\operatorname{End}(V_\mu^K))^{K\times K}$, i.e., $(K\times K)$-invariant elements. So $\Phi\colon G\rightarrow\operatorname{End}\big(V_\mu^K\big)$ is a matrix spherical function if
\[\Phi(k_1g k_2)=\pi_\mu^K(k_1)\Phi(g)\pi_\mu^K(k_2)\]
for any $k_1,k_2\in K$ and $g\in G$.
\end{Definition}

Assuming that $(G,K,\mu)$ is a multiplicity free triple, we can associate a matrix spherical function to each $\lambda\in P_G^+(\mu)$.

{\samepage

\begin{Lemma}\label{newdef}
For $\lambda\in P_{G}^{+}(\mu)$, we define a matrix-valued function $\Phi\colon G\rightarrow\operatorname{End}\big(V_{\mu}^{K}\big)$ such that
\[\Phi^{\mu}_{\lambda}(g)=j^{\ast}\circ\pi_{\lambda}^{G}(g)\circ j,\]
where $j$ is a $K$-invariant embedding, i.e., $j\in\textnormal{Hom}_K\big(V_\mu^K,V_\lambda^G\big)$, and $j^\ast$ is the adjoint of $j$. Then~$\Phi^{\mu}_{\lambda}$ is a matrix spherical function.
\end{Lemma}
\begin{proof}
See \cite[p.~15]{Prui-IMRN}.
\end{proof}}

\begin{Remark}
 Note that Definition \ref{def} includes other spherical functions beyond the ones considered in Lemma \ref{newdef}.
\end{Remark}

We recall \cite[Section~2]{KoelvPR-JFA}. Since all
$K$ representations are unitary, we can take $j$ to be unitary. Then $j^\ast\in\textnormal{Hom}_K\big(V_\lambda^G,V_\mu^K\big)$ and $j^\ast\circ j=\mathbb{I}_{V_\mu^K}$ and $j\circ j^\ast$ is a projection operator. Since $\dim_\mathbb{C}\textnormal{Hom}_K\big(V_\mu^K,V_\lambda^G\big)=1$, $\Phi_\lambda^\mu$ is independent of the choice of $j$.

We define such spherical functions in Lemma \ref{newdef} as the zonal spherical functions if $\mu=0$ and $\lambda\in P_{G}^{+}(0)$. We denote the vector space spanned by $\Phi_{\lambda}^{\mu}$'s with $\lambda\in P_{G}^{+}(\mu)$ by $E^{\mu}$, i.e., $E^\mu=\big(\mathbb{C}[G]\otimes\operatorname{End}\big(V_\mu^K\big)\big)^{K\times K}$, and the vector space spanned by zonal spherical functions by~$E^{0}$. Note that $E^\mu$ is a module over $E^0$; multiplying $\Phi\in E^\mu$ with $\phi\in E^0$ gives
\[(\phi\Phi)(k_1g k_2)=\pi_\mu^K(k_1)(\phi\Phi)(g)\pi_\mu^K(k_2)\]
since $\phi$ is bi-$K$-invariant.

\begin{Remark}\label{sphericalfunctionacting}
We choose a basis $v_1,v_2,\dots,v_d$ of $V_\mu^K$ with $d=\dim V_\mu^K$, then for $v\in V_{\mu}^{K}$, we have
\[\Phi_{\lambda}^{\mu}(g)v=\sum_{i=1}^{d} \frac{\big\langle\pi_{\lambda}^{G}(g)j(v),j(v_{i})\big\rangle}{\langle j(v_{i}),j(v_{i})\rangle}v_{i},\]
where $\langle\cdot,\cdot\rangle$ is a $G$-invariant Hermitian inner product.
\end{Remark}

\begin{Remark}[orthogonality]\label{sphfuncortho}
The Schur orthogonality relations for the matrix spherical functions give
\[\big\langle\Phi^{\mu}_{\lambda},\Phi^{\mu}_{\lambda'}\big\rangle=\int_{G} \operatorname{Tr}\big(\Phi^{\mu}_{\lambda}(g)\big(\Phi^{\mu}_{\lambda'}(g)\big)^{\ast}\big)\,{\rm d}g
=\frac{\bigl(\dim V_{\mu}^{K}\bigr)^{2}}{\dim V_{\lambda}^{G}}\delta_{\lambda,\lambda'}.\]
\end{Remark}

The focus of this paper is to calculate the spherical functions $\Phi^{\mu}_{\lambda}$ as explicitly as possible. To make such a function explicit we need to know the embedding $j$ in Lemma~\ref{newdef} explicitly, a~notoriously difficult problem. To narrow down the problem we make several assumptions:\looseness=1
\begin{itemize}\itemsep=0pt
\item Only particular irreducible $K$-representations are considered. The classification in \cite{PezzvP} gives roughly two families of irreducible $K$-representations, one for the first block of $K$ and one for the second.
\item The module structure of $E^\mu$ can be understood on a spectral level, where the spectrum has a product structure $B(\mu)\times\mathbb{N}^{n}$. The set $B(\mu)$ is called the bottom and in the transition from the spherical functions to the orthogonal polynomials, the crucial information is captured by the spherical functions with irreducible $G$-representations from the bottom.\looseness=1
\end{itemize}

Here is the main idea of this paper. Instead of calculating the spherical functions, the spherical functions $\Phi^{\mu}_{\lambda}$ are approximated by functions $Q^{\mu}_{\lambda}$. They are approximations in the following sense,
\[Q^{\mu}_{\lambda}=a_{\lambda}\Phi^{\mu}_{\lambda}\mbox{ $+$ lower order terms},\]
where $a_{\lambda}$ is a non-zero constant. The lower order terms can be described by the partial ordering on the weight lattice for $G$.

If the approximations are known, then the spherical functions can be recovered by means of an extra piece of information, namely that they are eigenfunctions of the quadratic Casimir operator. To fully control this operator, it has to be calculated explicitly which is technically involved.

\section{Structure theory}\label{sec:structhe}

The goal of this section is to describe the structure theory of the compact
symmetric space and to fix notation. We take $G=\mathrm{SU}(n+m)$ and $K=\mathrm{S}(\mathrm{U}(n)\times \mathrm{U}(m))$, where $m\geq n$ and $K$ is the $(2\times 2)$-block diagonal type with ${\rm U}(n)$ in the upper left-hand block and ${\rm U}(m)$ in the down right-hand block. This section is a generalization of \cite[Section 2.1]{KoelLiu}.

Let
\[
L=\left(
 \begin{matrix}
 0 & 0 & \cdots & 1 \\
 0 & \cdots & 1 & 0 \\
 \vdots & \vdots & \vdots & \vdots \\
 1 & 0 & \cdots & 0 \\
 \end{matrix}
 \right)
\] be an $n\times n$ matrix. The abelian subgroup $A$ of $G$ is given by
\begin{gather}
 A=\Bigg\{a_{\mathbf{t}}=\left(
 \begin{matrix}
 X & 0 & Y L \\
 0 & I_{m-n} & 0 \\
 L Y & 0 & L X L \\
 \end{matrix}
\right)\in G,\,
\mathbf{t}=(t_{1},t_{2},\dots,t_{n}), \, t_i\in\mathbb{R}\nonumber\\
\phantom{\mathbf{t}=\Bigg\{ }{}\mid X=\operatorname{diag}(\cos t_{1},\cos t_{2},\dots,\cos t_{n}),\, Y=\operatorname{diag}(\mathrm{i}\sin t_{1},\mathrm{i}\sin t_{2},\dots, \mathrm{i}\sin t_{n})\Bigg\},\label{atinA}
\end{gather}
where $I_{m-n}$ is an $(m-n)\times(m-n)$ identity matrix.

Also $M=Z_K(A)$ is given by
\begin{align*}
 M=\left\{\!\left(
 \begin{matrix}
 Z & 0 & 0 \\
 0 & M_1 & 0 \\
 0 & 0 & L Z L \\
 \end{matrix}
\right)\!\!\in\! K \mid Z=\operatorname{diag}\big({\rm e}^{\mathrm{i}t_1},{\rm e}^{\mathrm{i}t_2},\dots,{\rm e}^{\mathrm{i}t_n}\big),\,
 M_1\!\in\! U(m-n),\det(M)=1\!\right\}\!.\!\!
\end{align*}

The complexification of $G$ is denoted by $G^\mathbb{C}= \mathrm{SL}(m+n,\mathbb{C})$.
The maximal torus $T_{G^\mathbb{C}} \subset G^\mathbb{C}$ is the subgroup of diagonal matrices, and
similarly, the maximal torus $T_{G} \subset G$ is the subgroup of diagonal matrices. Also $K^\mathbb{C}$, $T_{K^\mathbb{C}}$, $M^\mathbb{C}$ and $T_{M^\mathbb{C}}$ are the corresponding complex type. Explicitly,
\begin{gather*}
T_{G^\mathbb{C}}=T_{K^\mathbb{C}}=\biggl\{\operatorname{diag}(t_1,t_2,\dots,t_{n+m})\mid t_i\in\mathbb{C},\prod_{i=1}^{n+m} t_i=1\biggr\},\\
T_{M^\mathbb{C}}=\{m=\operatorname{diag}(t_{1},t_{2},\dots,t_{n},t_{n+1},t_{n+2},\dots,t_{m},t_{n},t_{n-1},\dots,t_{1})\mid t_i\in\mathbb{C},\det(m)=1\}.
\end{gather*}

We define $\epsilon_{i}\colon T_{G^{\mathbb{C}}}\to\mathbb{C}^{\times}$ by $\epsilon_{i}(t)=t_{i}$. The holomorphic characters of $T_{G^{\mathbb{C}}}$ form an abelian group by pointwise multiplication and we use the additive notation for this group. For example,
\[-\epsilon_{i}(t)=t_{i}^{-1},\qquad (\epsilon_{i}-\epsilon_{j})(t)=t_{i}t_{j}^{-1}.\]
We define the orthogonality relation
\[\langle\epsilon_i,\epsilon_j\rangle=\delta_{ij}.\]

We define $\mathfrak{g}$, $\mathfrak{k}$, $\mathfrak{m}$ and $\mathfrak{a}$ as the corresponding complex Lie algebras of $G$, $K$, $M$ and $A$. In this notation $\Phi=\{\epsilon_{i}-\epsilon_{j}\mid 1\le i\ne j\le n+m\}$ is the set of roots of $\mathfrak{g}$ and it is of $A_{n+m-1}$ type. Following Bourbaki \cite[p.~250]{Bour}, we take $\Phi^{+}=\{\epsilon_{i}-\epsilon_{j}\mid 1\le i< j\le n+m\}$ as the set of positive roots for which $\Pi=\{\epsilon_{i}-\epsilon_{i+1}\mid 1\le i\le n+m-1\}$ is a system of simple roots. Note that~$\Phi=\Phi^{+}\cup\Phi^{-}$ is a disjoint union, where $\Phi^{-}=-\Phi^{+}$. We denote $\alpha_{i}=\epsilon_{i}-\epsilon_{i+1}$. The root lattice is denoted by $Q=\bigoplus_{i=1}^{n+m-1}\mathbb{Z}\alpha_{i}$. Let $Q^{+}=\bigoplus_{i=1}^{n+m-1}\mathbb{N}\alpha_{i}$ denote the non-negative integral linear combinations of the simple roots. The partial ordering $\eta\preccurlyeq\tau$ is $\tau-\eta\in Q^+$.

We define the fundamental weights corresponding to these positive roots. We define $\omega_{i}$'s for~$i=1,2,\dots,m+n-1$ as
\[\omega_{i}(\operatorname{diag}(t_{1},t_{2},\dots,t_{m+n}))=t_{1}t_{2}\cdots t_{i},\qquad \operatorname{diag}(t_{1},t_{2},\dots,t_{m+n})\in T_{G^\mathbb{C}}=T_{K^\mathbb{C}}.\]
Then we have
\[P^+_K =\Bigg\{ \sum_{i=1}^{n+m-1} a_i \omega_i\mid a_i\in\mathbb{N}, \,i\neq n, \, a_n\in\mathbb{Z}\Bigg\},\qquad
 P^+_G = \bigoplus_{i=1}^{n+m-1} \mathbb{N} \omega_i.\]

Now we consider the highest weight of $M$-modules. We define $\eta_i$ as the characters of $T_{M^\mathbb{C}}$ by
\[\eta_{i}(\operatorname{diag}(t_{1},t_{2},\dots,t_{n},t_{n+1},t_{n+2},\dots,t_{m},t_{n},t_{n-1},\dots,t_{1}))=t_{1}t_{2}\cdots t_{i}\]
for $i=1,2,\dots,m-1$ and then $\eta_{i}=\omega_{i}|_{T_{M^\mathbb{C}}}$. Then the highest weight of irreducible $M$-modules can be written as
\[\sigma=a_{1}\eta_{1}+a_{2}\eta_{2}+\cdots+a_{m-1}\eta_{m-1}\]
with $a_{1},a_{2},\dots,a_{n}\in\mathbb{Z}$ and $a_{n+1},a_{n+2},\dots,a_{m-1}\in\mathbb{N}$. We define the set of the highest weight of $M$ as $P_{M}^{+}$.

\begin{Remark}\label{weightrootrelation}
We will use the relation given by \cite[p.~250]{Bour}, where
\[\omega_{i}=(\epsilon_{1}+\epsilon_{2}+\cdots+\epsilon_{i})-\frac{i}{m+n}\sum_{k=1}^{m+n} \epsilon_{k}\]
and
\[\rho=\frac{1}{2}\sum_{\alpha\in\Phi^+} \alpha=\sum_{i=1}^{m+n-1} \omega_{i}.\]
\end{Remark}

We consider the relation between the irreducible $K$-module and $G$-module. By Kobayashi~\cite[Theorem 30]{Koba}, see also Deitmar \cite[Theorem 3]{Deit}, we have
\begin{Lemma}\label{GKMmultifree}
For $\mu\in P_{K}^{+}$, if $V_{\mu}^{K}|_{M}$ splits multiplicity free, then $\big[V_{\lambda}^{G}|_{K}:V_{\mu}^{K}\big]\leq1$ for all~${\lambda\in P_{G}^{+}}$.
\end{Lemma}
In this paper, we always consider the situation of $V_\mu^K|_M$ splitting multiplicity free.

\subsection[The case of mu=a omega\_1+b omega\_n]{The case of $\boldsymbol{\mu=a\omega_1+b\omega_n}$}\label{caseofaomega1bomegan}
\begin{Lemma}\label{deitgkmmultifreeaomega1}
Let $\mu=a\omega_{1}+b\omega_{n}\in P_{K}^{+}$ for $a,b\in\mathbb{N}$, then $\big[V_{\lambda}^{G}|_{K}:V_{\mu}^{K}\big]\leq1$ for all $\lambda\in P_{G}^{+}$. Moreover, we have $V_{\mu}^{K}\cong S^{a}\mathbb{C}^{n}\otimes \big(\bigwedge^{n}\mathbb{C}^{n}\big)^{\otimes b}$, and each weight vector of $V_{\mu}^{K}$
\[v_{\sigma}=x_{1}^{a_{1}}x_{2}^{a_{2}}\cdots x_{n}^{a_{n}}\otimes(e_{1}\wedge e_{2}\wedge\cdots\wedge e_{n})^{\otimes b},\qquad a_{1}+a_{2}+\cdots+a_{n}=a\]
generates a $1$-dimensional $M$-module, which corresponds to inequivalent $M$-types.
\end{Lemma}
\begin{Remark}
The first statement is also proved in \cite[p.~26, Table B.2.2]{PezzvP} by the case $P=\{\beta_1\}$.
\end{Remark}
\begin{proof}
For $m\in M$ of the form
\[\left(
 \begin{matrix}
 \operatorname{diag}(t_{1},t_{2},\dots,t_{n}) & 0 \\
 0 & \ast \\
 \end{matrix}
\right)
=m\in M,\]
we have
\[\pi(m)v_{\sigma}=t_{1}^{a_{1}+b}t_{2}^{a_{2}+b}\cdots t_{n}^{a_{n}+b}v_{\sigma}.\]
So $v_{\sigma}$ generates a $1$-dimensional $M$-module and the corresponding characters are all different. The result follows from Lemma \ref{GKMmultifree}.
\end{proof}

Now we calculate $P_G^+(\mu)$. The case of $n=2$ is given by \cite{KoelLiu}. By the case $P=\{\beta_1\}$ in \cite[p.~26, Table B.2.2]{PezzvP} with $n>2$, we define
\begin{gather*}
\xi_{i}=(\lambda_{i},0),\quad i=1,2,\dots,n-1,\\ \xi_{n}=(\omega_{1},\omega_{1}),\qquad\xi_{n+1}=(\omega_{2}+\omega_{m+n-1},\omega_{1}),\qquad\dots,\qquad
\xi_{2n-2}=(\omega_{n-1}+\omega_{m+2},\omega_{1}),\\
\xi_{2n-1}=(\omega_{m+1},\omega_1-\omega_n),\qquad\xi_{2n}=(\omega_{m},-\omega_{n}),\qquad\xi_{2n+1}=(\omega_{n},\omega_{n}).
\end{gather*}
Pezzini and van Pruijssen \cite{PezzvP} define the extended weight monoid, which consists of the pairs
\[(\tau_{1},\tau_{2})=\sum_{i=1}^{2n+1} c_{i}\xi_{i},\qquad c_{i}\in\mathbb{N}.\]
So $P^+_G(\mu)$ consists of those pairs $(\tau_1,\tau_{2})$ for which $\tau_2=\mu=a\omega_1+b\omega_n$. Note that we have multiplied the second entry of the pairs of the extended weight monoid by $-1$ in comparison to~\cite{PezzvP}.

For the special case $\mu=0$ we obtain the description of the spherical representation $P_{G}^{+}(0)=\bigoplus_{i=1}^{n} \mathbb{N}(\omega_i+\omega_{n+m-i})$ for symmetric pair $(G,K)$ as proved by Kr\"{a}mer \cite{Kram}. Note that $\lambda_i=\omega_i+\omega_{n+m-i}$ with $i=1,2,\dots,n$ are the generators of the spherical weights and they can be written in terms of simple roots as follows:
\begin{gather*}
\lambda_{1}=\omega_{1}+\omega_{m+n-1}=\sum_{j=1}^{m+n-1} \!\alpha_{j},\qquad
\lambda_{i}=\omega_{i}+\omega_{m+n-i}=\lambda_{i-1}+\sum_{j=i}^{m+n-i} \!\alpha_{j},\quad i=2,3,\dots,n.
\end{gather*}
Let $\lambda_{\rm sph}=\sum_{i=1}^{n}d_{i}\lambda_i\in P_G^+(0)$, we define $|\lambda_{\rm sph}|=\sum_{i=1}^{n}d_{i}$.

There is a trivial representation $V_{\omega_{0}}^{K}\subset V_{\omega_{i}+\omega_{m+n-i}}^{G}$ for $i=0,1,\dots,n$ with convention $\omega_0=\omega_{n+m}=0$ such that $\big[V_{\omega_{i}+\omega_{m+n-i}}^{G}|_{K},V_{\omega_{0}}^{K}\big]=1$, and we define the corresponding vector $v_{i}\in V_{\omega_{0}}^{K}$ as the $K$-fixed vector. Also we rewrite $\Phi^{0}_{\lambda_{i}}$ as $\phi_{i}$ and it is bi-$K$-invariant.
\begin{Proposition}\label{prop:P+Gmuofrankn} For the multiplicity free triple $(G,K, \mu)$ with $\mu=a\omega_{1}+b\omega_{n}$, $a,b\in \mathbb{N}$, we have $P^+_G(\mu)=B(\mu) + P^+_G(0)$, where
\[B(\mu) =
\Bigg\{ \nu = \sum_{i=1}^n a_{i} (\omega_{i}+\omega_{m+n+1-i})+b\omega_{n}
\mid \sum_{i=1}^n a_{i}=a,\, a_{i}\in\mathbb{N}\Bigg\}.\]
with the convention $\omega_{0}=\omega_{m+n}=0$.
\end{Proposition}
\begin{proof}
We prove this proposition by straightforward calculation to find out all the $\lambda\in P_{G}^{+}$ satisfying
\[(\lambda,a\omega_{1}+b\omega_{n})=\sum_{i=1}^{2n+1} c_{i}\xi_{i},\qquad c_{i}\in\mathbb{N}.\]
It leads to $c_i$ being free for $1\leq i\leq n-1$, which corresponds to $\sum_{i=1}^{n-1}c_i\lambda_i$, a spherical weight. The non-trivial remaining conditions are
\[\sum_{i=n}^{2n-1} c_i=a,\qquad
-c_{2n-1}-c_{2n}+c_{2n+1}=b. \]
Since $\xi_{2n}+\xi_{2n+1}=(\lambda_n,0)$ leads to the remaining generator of the spherical weights $P_G^+(0)$, we can additionally assume $c_{2n}c_{2n+1}=0$ in order to determine $B(\mu)$.

Since $b\in\mathbb{N}$ we need to take $c_{2n}=0$ and $B(\mu)$ is described by $\sum_{i=n}^{2n-1}c_i=a$, $c_{2n+1}=b+c_{2n-1}$. Relabelling gives the result and the proposition is proved.
\end{proof}
\begin{Remark}
Note that for any $\lambda\in P_G^+(\mu)$, we have $\lambda\succcurlyeq\mu$.
\end{Remark}
\subsection[The case of mu=omega\_s+b omega\_n]{The case of $\boldsymbol{\mu=\omega_s+b\omega_n}$}\label{subsec:omegas}
Let $1<s<n$ and $b\in\mathbb{N}$. The goal of this subsection is to give some preliminaries for the method to calculate $P_G^+(\omega_s+b\omega_n)$ in Section \ref{sec:omegasbomegan}.
\begin{Lemma}\label{deitgkmmultifree}
Let $\mu=\omega_{s}+b\omega_n\in P_{K}^{+}$, then $\big[V_{\lambda}^{G}|_{K}:V_{\mu}^{K}\big]\leq1$ for all $\lambda\in P_{G}^{+}$.
\end{Lemma}
\begin{proof}
Let $H=\{h_1,h_2,\dots,h_s\}\subset\{1,2,\dots,n\}$ be an ordered tuple with $h_i<h_{i+1}$. Since $V_{\omega_s+b\omega_n}^K\cong \bigwedge^s\mathbb{C}^n\otimes \big(\bigwedge^n\mathbb{C}^n\big)^{\otimes b}$, the weight vector in $V_\mu^K$ can be written as $e_{h_1}\wedge e_{h_2}\wedge\cdots\wedge e_{h_s}\otimes (e_1\wedge e_2\wedge\cdots\wedge e_n)^{\otimes b}$,
and it also generates a $1$-dimensional $M$-module. Then $V_{\mu}^{K}|_{M}$ splits multiplicity free. By Lemma \ref{GKMmultifree}, we have $\big[V_{\lambda}^{G}|_{K}:V_{\mu}^{K}\big]\leq1$ for any $\lambda\in P_{G}^{+}$ and the lemma follows.
\end{proof}

We have the stability result for the multiplicities due to van Pruijssen \cite{Prui-IMRN}.

\begin{Remark}\label{Bomega2induction}
Let $\lambda\in P_G^+$,
then in \cite[Section~3]{Prui-IMRN} it is shown that $\big[V_{\mu}^{K}|_{M}:V_{\lambda|_{T_{M}}}^{M}\big]\geq\big[V_{\lambda}^{G}|_{K}:V_{\mu}^{K}\big]$ and $\big[V_{\lambda+\omega_{i}+\omega_{m+n-i}}^{G}|_{K}:V_{\mu}^{K}]\geq[V_{\lambda}^{G}|_{K}:V_{\mu}^{K}\big]$.
\end{Remark}

For a weight vector $v_H=e_{h_1}\wedge e_{h_2}\wedge\cdots\wedge e_{h_s}\otimes (e_1\wedge e_2\wedge\cdots\wedge e_n)^{\otimes b}\in V_\mu^K$, the corresponding weight is
\[\eta_H=\sum_{i=1}^s (\omega_{h_i}-\omega_{h_i-1})+b\omega_n=\omega_{x_1}-\omega_{y_1}+\omega_{x_2}-\omega_{y_2}+\cdots+\omega_{x_o}-\omega_{y_o}+b\omega_n.\]
Note that in this equation, we have $y_1<x_1<y_2<x_2<\cdots<y_o<x_o$, $o\leq s$ and ${\sum_{i=1}^{o}(x_i-y_i)=s}$. By adding a spherical weight $\sum_{i=1}^{o}(\omega_{y_i}+\omega_{y_{m+n-i}})$ on $\eta_H$, we have
\begin{equation}\label{expressionoflambdaH}
 \lambda_H=\omega_{x_1}+\omega_{m+n-y_1}+\omega_{x_2}+\omega_{m+n-y_2}+\cdots+\omega_{x_o}+\omega_{m+n-y_o}+b\omega_n.
\end{equation}
Then we define a set
\[B'(\mu)=\{\lambda_H\mid H=\{h_1,h_2,\dots,h_s\}\subset \{1,2,\dots,n\},\, h_1<h_2<\cdots<h_s\}\]
and we have
\begin{Lemma}\label{lem:PG+mudec}
For any $\lambda\in P_G^+(\mu)$, we have
\[\lambda=\lambda_H+\lambda_{\rm sph},\qquad \lambda_H\in B'(\mu),\qquad \lambda_{\rm sph}\in P_G^+(0).\]
Moreover, $P_G^+(\mu)$ has the product structure $B'(\mu)\times\mathbb{N}^n$ as mentioned in Section~{\rm \ref{subsection:notationsandpreliminaries}}.
\end{Lemma}
\begin{proof}
For $\lambda\in P_G^+(\mu)$, since $V_{\mu}^K|_M$ splits multiplicity free, by Remark \ref{Bomega2induction} we have
\[1=\big[V_\lambda^G|_K:V_\mu^K\big]\leq\big[V_\mu^K|_M:V_{\lambda|_{T_M}}^M\big]=1.\]
It leads to
\[\lambda=\eta_H+\lambda_{\rm sph}',\qquad \lambda_{\rm sph}'\in P_G^+(0)\]
since $(\omega_i+\omega_{n+m-i})|_{T_M}=0$. Let
\[\lambda=\eta_H+\lambda_{\rm sph}'=\lambda_H+\sum_{i=1}^n d_i(\omega_i+\omega_{n+m-i}),\qquad d_i\in\mathbb{Z},\]
where
\[\lambda_{H}=\sum_{i=1}^{n+m} b_i\omega_i\in B'(\mu),\qquad b_i\in\mathbb{N},\]
then either $b_i=0$ or $b_{n+m-i}=0$ by \eqref{expressionoflambdaH}. We assume $b_i=0$. If $d_i<0$, then the coefficient of $\omega_i$ in $\lambda$ is negative, which contradicts the fact that $\lambda$ is dominant. Also the situation of $b_{n+m-i}=0$ is similar. So $d_i\geq0$ for $i=1,2,\dots,n$ and we have
\[\lambda=\lambda_H+\lambda_{\rm sph},\qquad \lambda_{\rm sph}\in P_G^+(0).\]
This lemma is proved.
\end{proof}

\section[Spherical function restricted to A]{Spherical function restricted to $\boldsymbol{A}$}\label{sec:sphfunctoA}
By Definition \ref{def} and the Cartan decomposition $G=KAK$, we know that the corresponding spherical functions are uniquely determined by the spherical function restricted to $A$. Let $\Phi\in E^\mu$, for $m\in M$ and $a\in A$, we have
\[\pi_{\mu}^{K}(m)\Phi(a)=\Phi(ma)=\Phi(am)=\Phi(a)\pi_{\mu}^{K}(m).\]
It leads to
\[\Phi\colon \ A\rightarrow\text{End}_{M}\big(V_{\mu}^{K}\big).\]
We have an $M$-module decomposition of $V_{\mu}^{K}$ such that
\[V_{\mu}^{K}|_{M}=\bigoplus_{l=1}^N V_{\rho_{l}}^{M},\]
then by Schur's lemma and $V_{\mu}^{K}|_{M}$ splitting multiplicity free, we have
\[\Phi(a)\in \operatorname{End}_{M}V_{\mu}^{K}\cong\bigoplus_{l=1}^N
\operatorname{End}_{M}\big(V_{\rho_{l}}^{M}\big)\cong\mathbb{C}^{N}.\]

\begin{Remark}\label{rem:ortho}
Let $\Phi,\Psi\in E^\mu$, then we have
\begin{align*}
\begin{split}
\langle\Phi,\Psi\rangle&=c_1\int_{K}\int_{A}\int_{K} \operatorname{Tr}(\Phi(k_{1}ak_{2})(\Psi(k_{1}ak_{2}))^{\ast})|\delta(a)|\, {\rm d}k_{1}{\rm d}a{\rm d}k_{2}\\
&=c_{1}\int_{A}\operatorname{Tr}(\Phi(a)(\Psi(a))^{\ast})|\delta(a)|{\rm d}a,
\end{split}
\end{align*}
where
\[c_{1}=\frac{1}{\int_{A}|\delta(a)|{\rm d}a}\]
by Cartan decomposition, since the integrand is independent of $k_{1}$ and $k_{2}$ and $\mu$ is a unitary representation. For $a=a_\mathbf{t}\in A$ as in \eqref{atinA}, the expression of $\delta(a)=\delta(a_{\mathbf{t}})$ is given by, see \cite[p.~383]{Helg-1962},
\[\delta(a_\mathbf{t})=\prod_{i=1}^n \sin^{2(m-n)}t_{i}\prod_{i=1}^n \sin 2t_{i}
\prod_{1\leq i<j\leq n}\bigl(\sin^{2}(t_{i}+t_{j})\sin^{2}(t_{i}-t_{j})\bigr).\]
\end{Remark}
Recall the restriction in Section \ref{subsection:notationsandpreliminaries}, where $E^\mu$ has a structure $B(\mu)\times\mathbb{N}^n$, we give a general structure of $\Phi_\lambda^\mu|_A$'s where $\lambda\in P_G^+(\mu)$.
\begin{Lemma}[{\cite[Lemma 6.1]{Prui-IMRN}}] $E^\mu$ is freely and finitely generated by $\phi_i$'s with $i=1,2,\dots,n$ as an $E^0$ module. Moreover, let $F_{\lambda}=\phi_{1}^{d_{1}}\phi_{2}^{d_{2}}\cdots\phi_{n}^{d_{n}}\Phi_{\nu}^{\mu}|_{A}$ with $\nu\in B(\mu)$ and $\lambda=\nu+\sum_{j=1}^{n}d_{j}\lambda_{j}\in P_{G}^{+}(\mu)$, then all $F_{\lambda}$'s are linearly independent.
\end{Lemma}
\begin{Lemma}\label{lem:sphericalpolynomials}
Let $\lambda\in P_{G}^{+}(\mu)$, then the corresponding spherical function can be written as
\[\Phi_{\lambda}^{\mu}|_{A}= \sum_{\lambda'\preccurlyeq\lambda}q_{\lambda'}F_{\lambda'},\qquad  q_{\lambda'}\in\mathbb{C}\]
and $q_{\lambda}\neq0$.
\end{Lemma}
\begin{proof}
It is true for all the elements in $B(\mu)$ since $F_\nu=\Phi_\nu^\mu|_A$ by definition. Then for $\nu+\sum_{i=1}^{n}(\omega_i+\omega_{n+m-i})=\lambda\in P_{G}^{+}(\mu)$, we assume it is true for all $P_{G}^{+}(\mu)\ni\lambda'\precneqq\lambda$. We define
\[{\rm U}_\lambda^G=V_{\nu}^G\otimes\bigotimes_{i=1}^n  \big(V_{\omega_i+\omega_{n+m-i}}^G\big)^{\otimes d_i},\qquad
 u=v_{\nu}\otimes \bigotimes_{i=1}^n  v_i^{\otimes d_i},\]
where $v_\nu$ is the $K$ highest weight vector with weight $\mu$ in $V_\nu^G$, and $v_i$'s are the $K$-fixed vectors in $V_{\omega_i+\omega_{n+m-i}}^G$, i.e., the vectors generating the trivial $K$ module as described in Section~\ref{caseofaomega1bomegan}. So~$u$ is also a $K$ highest weight vector with weight $\mu$ in ${\rm U}_\lambda^G$, and $F_\lambda$ is the corresponding matrix spherical function. Since
\[{\rm U}_{\lambda}^{G}=V_{\lambda}^{G}\oplus \bigoplus_{\lambda'\precneqq\lambda}m_{\lambda'}V_{\lambda'}^{G}\]
by the complete reducibility theorem, we have
\[F_{\lambda}=d_{\lambda}\Phi_{\lambda}^{\mu}|_{A}+ \sum_{\lambda'\precneqq\lambda}d_{\lambda'}\Phi_{\lambda'}^{\mu}|_{A}
=d_{\lambda}\Phi_{\lambda}^{\mu}|_{A}+\sum_{\lambda'\precneqq\lambda}d'_{\lambda'}F_{\lambda'}\]
by the induction hypothesis. If $d_{\lambda}=0$, then $F_{\lambda}$ can be written as
\[F_{\lambda}=\sum_{\lambda'\precneqq\lambda}d'_{\lambda'}F_{\lambda'},\]
which contradicts the fact that $F_{\lambda}$'s are linearly independent. So $d_{\lambda}\neq0$ and this lemma is proved.
\end{proof}

\section{Radial part of the Casimir operator}\label{sec:radialpartofthecasimiroperator}
Now we give the explicit expression of the radial part $R$ of the Casimir operator in this case for the matrix spherical functions related to the $K$-representation. We follow the approach of Casselman and Mili\v{c}i\'{c} \cite{CassM}, see also Warner \cite[Proposition 9.1.2.11]{Warn-2}. The meaning of the radial part is that such an operator is acting on the function restricted to $A\subset G$. By \cite[Section~2.2]{KoelvPR-JFA}, we know that the spherical function restricted to $A$ is the eigenfunction of the radial part $R$ of the Casimir operator and
\[R\big(\Phi_{\lambda}^{\mu}|_{A}\big)=c_{\lambda}\Phi_{\lambda}^{\mu}|_{A}\]
with $c_{\lambda}=\langle\lambda,\lambda\rangle+2\langle\lambda,\rho\rangle$. We give the explicit calculation in Appendix \ref{app:caloper}.

We need to prove the following lemma before we calculate spherical function $\Phi_{\lambda}^{\mu}|_{A}$'s.
\begin{Lemma}\label{eigveldiff}
For $\lambda\succneqq\lambda'$ with $\lambda,\lambda'\in P_{G}^{+}(\mu)$, we have $c_\lambda>c_{\lambda'}$.
\end{Lemma}
\begin{proof}
We have
\begin{align*}
c_{\lambda}-c_{\lambda'}&=\langle\lambda,\lambda\rangle+2\langle\lambda,\rho\rangle-\langle\lambda',\lambda'\rangle-2\langle\lambda',\rho\rangle=\langle\lambda+\rho,\lambda+\rho\rangle-\langle\lambda'+\rho,\lambda'+\rho\rangle\\
&=\langle\lambda+\lambda'+2\rho,\lambda-\lambda'\rangle.
\end{align*}
Since $\lambda+\lambda'+2\rho\in P_{G}^{+}$ and $\lambda-\lambda'\in Q^{+}$, it leads to $c_{\lambda}>c_{\lambda'}$ and this lemma is proved.
\end{proof}

Now fix $\lambda\in P_{G}^{+}(\mu)$, then the finite-dimensional space of matrix spherical functions spanned by $\big\{\Phi_{\lambda'}^{\mu}|_A\big\}_{\lambda'\preccurlyeq\lambda}$ is also spanned by $\{F_{\lambda'}\}_{\lambda'\preccurlyeq\lambda}$. Moreover, by Lemma \ref{lem:sphericalpolynomials}, the transition between these bases is triangular. Since the basis $\big\{\Phi_{\lambda'}^{\mu}|_A\big\}_{\lambda'\preccurlyeq\lambda}$ is a basis of eigenfunctions for the action of the radial part $R$ of the Casimir operator $R$, the space is invariant for $R$. By Lemma \ref{eigveldiff}, the eigenspace for the eigenvalue $c_{\lambda}$ is one-dimensional. Using Lemma \ref{lem:sphericalpolynomials}, we find that $R$ acts lower triangularly on $F_{\lambda}$,
\[R(F_{\lambda})=c_{\lambda}F_{\lambda}+\sum_{\lambda'\precneqq\lambda}b_{\lambda'}F_{\lambda'}.\]
From this result, we can obtain precise information on the matrix spherical functions $\Phi_\lambda^\mu|_A$.

\section{Special cases}\label{geninj}
The goal of this section is to give some simple cases of the matrix spherical functions for the multiplicity free triple $(G,K,\omega_s+b\omega_n)$ with $s=0,1,\dots,n$ and $b\in\mathbb{N}$. This will be used in Sections \ref{exa:aomega1+bomegan} and \ref{sec:omegasbomegan} to calculate the approximate functions.

\begin{Theorem}\label{gdec}
For $c\leq n+m-d$, we have
\[V_{\omega_{c}}^{G}\otimes V_{\omega_{n+m-d}}^{G}\cong\bigoplus_{i=0}^{\min\{c,d\}} V_{\omega_{c-i}+\omega_{n+m-d+i}}^{G},\]
where $\omega_0=\omega_{n+m}=0$ by convention.
\end{Theorem}
\begin{proof}
The proof of this theorem is similar to \cite[Lemma 3.1]{KoelLiu} by using \cite[Proposition 1]{Kass} and~\cite[Corollary 3.5]{Kuma}.
\end{proof}

For the ordered tuple $H=\{h_{1},h_{2},\dots,h_{s}\}$ with $h_{i}<h_{i+1}$ as a subset of $\{1,2,\dots,n\}$, we define $e_{H}=e_{h_{1}}\wedge e_{h_{2}}\wedge\cdots\wedge e_{h_{s}}$ and $\cos t_{H}=\cos t_{h_{1}}\cos t_{h_{2}}\cdots\cos t_{h_{s}}$.

\begin{Corollary}\label{kmodulenumber}
We have $\big[V_{\omega_{s+u}}^{G}\otimes V_{\omega_{m+n-u}}^{G}|_{K}:V_{\omega_{s}}^{K}\big]\leq u+1$.
\end{Corollary}
\begin{proof}
By Theorem \ref{gdec}, we have
\[V_{\omega_{s+u}}^{G}\otimes V_{\omega_{m+n-u}}^{G}\cong\bigoplus_{p=0}^{u} V_{\omega_{s+u-p}+\omega_{m+n-u+p}}^{G}.\]
Also by Lemma \ref{deitgkmmultifree} it leads to $\big[V_{\omega_{s+u-p}+\omega_{m+n-p}}^{G}|_{K}:V_{\omega_{s}}^{K}\big]\leq1$ implying the corollary.
\end{proof}
\begin{Remark}\label{laplace}
We recall the Laplace expansion in \cite[p.~22]{Gant}. Define $\xi_{I}^{J}$ as a minor of a $p\times p$ matrix ${\rm U}$ for which $I=\{i_{1},i_{2},\dots,i_{t}\}$ and $J=\{j_{1},j_{2},\dots,j_{t}\}$ are ordered tuples. For a $t$-tuple $I=\{i_{1},i_{2},\dots,i_{t}\}$, the $(p-t)$-tuple $\bar{I}$ is the ordered tuple so that $I\cup\bar{I}=\{1,2,\dots,p\}$. Then for any fixed $t$ and denoting $b(I)=\sum_{q=1}^{t}i_{q}$, we have
\begin{equation}\label{laplaciandet}
\sum_{|J|=t} (-1)^{b(J)}\xi_{J}^{I}\xi_{\bar{J}}^{\bar{K}}
=(-1)^{b(I)}\delta_{IK}\det({\rm U}).
\end{equation}
\end{Remark}
For ease of expression, we define $\mathrm{N}=\{1,2,\dots,n\}$ and $\mathrm{M}=\{n+1,n+2,\dots,n+m\}$. For $P=\{p_{1},p_{2},\dots,p_{|P|}\}\subset \mathrm{N}$ with $p_{i}<p_{i+1}$, we define $e_{P}=e_{p_{1}}\wedge e_{p_{2}}\wedge\cdots\wedge e_{|P|}$, where $|P|$ is length of $P$. We use the notation $\mathrm{N}\setminus P$ as ordered pair of length $n-|P|$ and $e_{\mathrm{N}\setminus P}$ is defined analogously. We use the same notation for $e_{Q}$ with $Q\subset \mathrm{M}$ and $e_{\mathrm{M}\setminus Q}$. Let $\pi(g)$ with~$g\in G$ be the standard representation of $G$, then we have $\pi(k)v=\pi(k_1,k_2)v$ with $v\in V_{\mu}^{K}$ for $k=\left(\begin{smallmatrix}
k_1&0\\0&k_2\\
\end{smallmatrix}\right)\in \mathrm{S}(\mathrm{U}(n)\times \mathrm{U}(m))\subset G$.
\begin{Theorem}\label{kintertwiner}
For $s=0,1,2,\dots,n$ and $u=0,1,\dots,n-s$, let
\begin{equation}\label{minj}
 v^{H,k}=\sum_{\substack {P\subset \mathrm{N},\,|P|=k\\ Q\subset \mathrm{M},\,|Q|=u-k}} (-1)^{b(P)+b(Q)}e_{H}\wedge e_{P}\wedge e_{Q}\otimes e_{\mathrm{N}\backslash P}\wedge e_{\mathrm{M}\backslash Q}\in
 V_{\omega_{s+u}}^{G}\otimes V_{\omega_{m+n-u}}^{G},
\end{equation}
where $H=\{h_{1}<h_{2}<\cdots<h_{s}\}\subset \mathrm{N}$ and $k=0,1,\dots,u$. We define a linear map
\begin{align*}
h^{k}\colon \ V_{\omega_{s}}^{K}\cong {\bigwedge}^{s}\mathbb{C}^{n}&\rightarrow  V_{\omega_{s+u}}^{G}\otimes V_{\omega_{m+n-u}}^{G},\\
e_{H}&\mapsto  v^{H,k},
\end{align*}
then $h^{k}$ is a $K$-interwiner.
\end{Theorem}

\begin{proof}
We have
\begin{align}
\big(\pi(k_{1},k_{2})h^{k}\big)(e_{H})&=\pi(k_{1},k_{2})\Bigg(\sum_{\substack {P\subset \mathrm{N}\\ Q\subset \mathrm{M}}} (-1)^{b(P)+b(Q)}e_{H}\wedge e_{P}\wedge e_{Q}\otimes e_{\mathrm{N}\backslash P}\wedge e_{\mathrm{M}\backslash Q}\Bigg)\nonumber\\
&=\sum_{P\subset \mathrm{N}} (-1)^{b(P)}\sum_{Q\subset \mathrm{M}} (-1)^{b(Q)}\nonumber\\
&\quad\times\pi(k_{1})e_{H}\wedge\pi(k_{1})e_{P}\wedge \pi(k_{2})e_{Q}\otimes \pi(k_{1})e_{\mathrm{N}\backslash P}\wedge \pi(k_{2})e_{\mathrm{M}\backslash Q}.\label{laplacianstep}
\end{align}
Since
\begin{equation*}
\pi(k_{2})e_{Q}=\sum_{\substack{R\subset \mathrm{M}\\|R|=|Q|}} \xi_{Q}^{R}(k_{2})e_{R},\qquad
\pi(k_{2})e_{\mathrm{M}\backslash Q}=\sum_{\substack{{\rm U}\subset \mathrm{M}\\|{\rm U}|=|\mathrm{M}\backslash Q|}} \xi_{\mathrm{M}\backslash Q}^{{\rm U}}(k_{2})e_{{\rm U}},
\end{equation*}
by the Laplace expansion of Remark \ref{laplace}, we see that \eqref{laplacianstep} equals
\begin{align*}
&\sum_{P\subset \mathrm{N}} (-1)^{b(P)}\sum_{Q\subset \mathrm{M}} (-1)^{b(Q)}\sum_{\substack{R\subset \mathrm{M}\\|R|=|Q|}} \xi_{Q}^{R}(k_{2})\sum_{\substack{{\rm U}\subset \mathrm{M}\\|{\rm U}|=|\mathrm{M}\backslash Q|}} \xi_{\mathrm{M}\backslash Q}^{{\rm U}}(k_{2})\\
&\qquad\quad{}\times\pi(k_{1})e_{H}\wedge\pi(k_{1})e_{P}\wedge e_{R}\otimes \pi(k_{1})e_{\mathrm{N}\backslash P}\wedge e_{{\rm U}}\\
&\qquad =\sum_{P\subset \mathrm{N}} (-1)^{b(P)}\sum_{\substack{R\subset \mathrm{M}\\|R|=|Q|}} \sum_{\substack{{\rm U}\subset \mathrm{M}\\|{\rm U}|=|\mathrm{M}\backslash Q|}} \Bigg(\sum_{Q\subset \mathrm{M}} (-1)^{b(Q)}\xi_{Q}^{R}(k_{2})\xi_{\mathrm{M}\backslash Q}^{{\rm U}}(k_{2})\Bigg)\\
&\qquad\quad{}\times\pi(k_{1})e_{H}\wedge\pi(k_{1})e_{P}\wedge e_{R}\otimes \pi(k_{1})e_{\mathrm{N}\backslash P}\wedge e_{{\rm U}}\\
&\qquad {}=\det(\pi(k_{2}))\sum_{P\subset \mathrm{N}} (-1)^{b(P)}\sum_{R\subset \mathrm{M}} (-1)^{b(R)}
\pi(k_{1})e_{H}\wedge\pi(k_{1})e_{P}\wedge e_{R}\otimes \pi(k_{1})e_{\mathrm{N}\backslash P}\wedge e_{\mathrm{M}\backslash R}
\end{align*}
using \eqref{laplaciandet}.

By a similar calculation and using the Laplace expansion for the action of $k_{1}$, we have the conclusion after renaming
\begin{gather*}
\det(\pi(k_{1}))\det(\pi(k_{2}))\sum_{P\subset \mathrm{N}} (-1)^{b(P)}\sum_{Q\subset \mathrm{M}} (-1)^{b(Q)}\pi(k_{1})e_{H}
\wedge e_{P}\wedge e_{Q}\otimes e_{\mathrm{N}\backslash P}\wedge e_{\mathrm{M}\backslash Q}\\
\quad =\sum_{\substack{P\subset \mathrm{N}\\ Q\subset \mathrm{M}}} (-1)^{b(P)+b(Q)}\pi(k_{1})e_{H}\wedge e_{P}\wedge e_{Q}\otimes e_{\mathrm{N}\backslash P}\wedge e_{\mathrm{M}\backslash Q}.
\end{gather*}
On the other hand, we have
\begin{gather*}
\big(h^{k}\pi(k_{1},k_{2})\big)(e_{H})
=h^{k}(\pi(k_{1})e_{H})
=\sum_{\substack{P\subset \mathrm{N}\\ Q\subset \mathrm{M}}} (-1)^{b(P)+b(Q)}\pi(k_{1})e_{H}\wedge e_{P}\wedge e_{Q}\otimes e_{\mathrm{N}\backslash P}\wedge e_{\mathrm{M}\backslash Q}.
\end{gather*}
Then $h^{k}$ is a $K$-intertwiner and this theorem is proved.
\end{proof}
\begin{Corollary}\label{cor:coe}
For any $H,H'\subset N$ with $|H|=|H'|=s$ and any $v^{H,k},v^{H',k'}\in V_{\omega_{s+u}}^{G}\otimes V_{\omega_{m+n-u}}^{G}$, we have
\[\big\langle v^{H,k},v^{H',k'}\big\rangle=\delta_{H,H'}\delta_{k,k'}\left(
 \begin{matrix}
 n-s \\
 k \\
 \end{matrix}
\right)\left(
 \begin{matrix}
 m \\
 u-k \\
 \end{matrix}
\right).\]
\end{Corollary}

\begin{Lemma}\label{lem:linearlycombinationofkmodule}
Each weight vector in $V_{\omega_{s}}^{K}\subset V_{\omega_{s+u}+\omega_{m+n-u}}^{G}\subset V_{\omega_{s+u}}^{G}\otimes V_{m+n-u}^{G}$ ($u=0,1,\dots,n-s$) can be written as linear combination of $v^{H,k}$ for $k=0,1,\dots,u$ in \eqref{minj}.
\end{Lemma}
\begin{proof}
By Corollary \ref{kmodulenumber}, we have the conclusion
\[\big[\big(V_{\omega_{s+u}}^{G}\otimes V_{m+n-u}^{G}\big)|_{K}:V_{\omega_{s}}^{K}\big]\leq u+1.\]
In Theorem \ref{kintertwiner}, we have found $u+1$ irreducible $K$-modules with highest weight $\omega_{s}$ in $V_{\omega_{s+u}}^{G}\otimes V_{\omega_{m+n-u}}^{G}$, so this lemma is proved.
\end{proof}
\begin{Remark}\label{coefficients}
Now we calculate the corresponding coefficients of Lemma \ref{lem:linearlycombinationofkmodule}. For $u=0$, we have
\[v_{0}^{H}=e_{h_{1}}\wedge e_{h_{2}}\wedge\cdots\wedge e_{h_{s}}\in V_{\omega_{s}}^{K}\subset V_{\omega_{s}}^{G}\]
using the standard embedding $\mathbb{C}^{n}\hookrightarrow\mathbb{C}^{n+m}$ and $\bigwedge^{s}\mathbb{C}^{n}\hookrightarrow\bigwedge^{s}\mathbb{C}^{n+m}$.

We have a $G$-intertwiner $\rho^{0}\colon V_{\omega_s}^G\rightarrow V_{\omega_{s+1}}^G\otimes V_{\omega_{n+m-1}}^G$ such that
\begin{align*}
\rho^{0}\big(v_{0}^{H}\big)={}&\rho^{0}\big(v_{0}^{H}\otimes e_{1}\wedge e_{2}\wedge\cdots\wedge e_{m+n}\big)\\
={}&\sum_{k=1}^{n}(-1)^{k}e_{h_{1}}\wedge e_{h_{2}}\wedge\cdots\wedge e_{h_{s}}\wedge e_{k}\otimes e_{1}\wedge e_{2}\wedge\cdots\wedge\widehat{e_{k}}\wedge\cdots\wedge e_{m+n}\\
&{}+\sum_{k=n+1}^{n+m}(-1)^{k}e_{h_{1}}\wedge e_{h_{2}}\wedge\cdots\wedge e_{h_{s}}\wedge e_{k}\otimes e_{1}\wedge e_{2}\wedge\cdots\wedge\widehat{e_{k}}\wedge\cdots\wedge e_{m+n}\\
={}& v^{H,1}+v^{H,0},
\end{align*}
where we use the notation of Theorem \ref{kintertwiner} and $\bigwedge^{n+m}\mathbb{C}^{n+m}$ being the trivial representation. So we have
\[\rho^{0}\big(v_{0}^{H}\big)\in V_{\omega_{s}}^{K}\subset V_{\omega_{s}}^{G}\subset V_{\omega_{s+1}}^{G}\otimes V_{\omega_{m+n-1}}^{G}.\]
Since, by Theorem \ref{gdec},
\[V_{\omega_{s+1}}^{G}\otimes V_{\omega_{m+n-1}}^{G}\cong V_{\omega_{s}}^{G}\oplus V_{\omega_{s+1}+\omega_{m+n-1}}^{G},\]
we know that there is an irreducible $K$-module with highest weight $\omega_{s}$ in $V_{\omega_{s+1}+\omega_{m+n-1}}^{G}$ by Lemma \ref{lem:linearlycombinationofkmodule}. Also the corresponding vector can be written as
\[v_{1}^{H}=c_{1}v^{H,1}+c_{0}v^{H,0}\]
and $v_{1}^{H}$ is orthogonal with all the vectors in $V_{\omega_{s}}^{G}$. So
\[\big\langle v_{0}^{H},v_{1}^{H}\big\rangle=(n-s)c_{1}+mc_{0}=0.\]
Then we can find the $K$-module with highest weight $\omega_{s}$ in $V_{\omega_{s+1}+\omega_{m+n-1}}^{G}$. This can be done by induction for more general $u$, but there seem no nice explicit expressions available.
\end{Remark}

Now we calculate the matrix-valued spherical functions. We consider the corresponding approximate function at first.

\begin{Definition}
For $\mu=\omega_s+b\omega_n$ and $V_{\mu}^{K}\cong\bigwedge^{s}\mathbb{C}^{n}\otimes\bigl(\bigwedge^{n}\mathbb{C}^{n}\bigr)^{\otimes b}$, let $\nu_i=\omega_{s+i}+\omega_{n+m-i}+b\omega_n$ and
\[W_{\nu_{i}}^{G}=V_{\omega_{s+i}}^{G}\otimes V_{\omega_{m+n-i}}^{G}\otimes \big(V_{\omega_{n}}^{G}\big)^{\otimes b},\]
then we define a map
\begin{align}
j_W\colon \ V_{\mu}^{K}&\rightarrow  W_{\nu_{i}}^{G},\nonumber\\
e_{H}\otimes (e_{\mathrm{N}})^{\otimes b}&\mapsto
\left(
 \begin{matrix}
 n-s \\
 i \\
 \end{matrix}
\right)^{-\frac{1}{2}}
\sum_{\substack{P\subset \mathrm{N}\\|P|=i}} (-1)^{b(P)}e_{H}\wedge e_{P}\otimes e_{\mathrm{N}\backslash P}\wedge e_\mathrm{M}\otimes (e_{\mathrm{N}})^{\otimes b}.\label{weightvectoreta}
\end{align}
Also we define the approximate function
\begin{equation}\label{approfuncQi}
Q^\mu_{\nu_{i}} \colon \ A \to \operatorname{End}(V^K_\mu),
\qquad
a\mapsto
 \left(
 \begin{matrix}
 n-s \\
 i \\
 \end{matrix}
\right)\cdot j_W^\ast \circ \pi_{W_{\nu_{i}}^{G}}(a) \circ j_W.
\end{equation}
\end{Definition}

\begin{Lemma}
The map $j_W$ is a unitary $K$-intertwiner.
\end{Lemma}

\begin{proof}
This follows from Theorem \ref{kintertwiner} and Corollary \ref{cor:coe} with $u=k=i$, and $\bigwedge^{n}\mathbb{C}^{n}\mapsto V_{\omega_{n}}^{G}\cong\bigwedge^{n}\mathbb{C}^{n+m}$ being a $K$-intertwiner.
\end{proof}

\begin{Remark}\label{rem:extendmultifree}
By Lemma \ref{deitgkmmultifree}, we know that $e_{H}\otimes (e_{\mathrm{N}})^{\otimes b}$ is a weight vector of $V_{\mu}^{K}$. Also $e_{H}\otimes (e_{\mathrm{N}})^{\otimes b}$ generates a one-dimensional $M$-module and $V_{\mu}^{K}|_{M}$ splits multiplicity free.
\end{Remark}

\begin{Lemma}\label{Qomegas+bomeganentry}
Let $a_{\mathbf{t}}\in A$ be as in \eqref{atinA}, then $Q^\mu_{\nu_{i}}(a_{\mathbf{t}})$ is a diagonal matrix-valued function and
\[\left(
 \begin{matrix}
 n-s \\
 i \\
 \end{matrix}
\right)\frac{\big\langle\pi(a_{\mathbf{t}})j_W\big(e_{H}\otimes(e_{N})^{\otimes b}\big),j_W\big(e_{H'}\otimes(e_{N})^{\otimes b}\big)\big\rangle}{\big\langle j_W\big(e_{H}\otimes(e_{N})^{\otimes b}\big),j_W\big(e_{H}\otimes(e_{N})^{\otimes b}\big)\big\rangle}
=\delta_{HH'}\cos t_{H}\cos^{b} t_{N}\sum_{\substack{I\in N\backslash H\\|I|=i}}  \cos^{2} t_{I}.\]
Moreover, the entry corresponding to $v=e_{H}\otimes(e_{N})^{\otimes b}$ is $\cos t_{H}\cos^{b} t_{N}\sum_{\substack{I\in N\backslash H\\|I|=i}} \cos^{2} t_{I}$.
\end{Lemma}
\begin{proof}
By Remark \ref{rem:extendmultifree} and $Q_{\nu_{i}}^{\mu}(a_{\mathbf{t}})$ being an $M$-intertwiner, we find
\[\big\langle\pi(a_{\mathbf{t}})j_W\big(e_{H}\otimes(e_{N})^{\otimes b}\big),j_W\big(e_{H'}\otimes(e_{N})^{\otimes b}\big)\big\rangle=0\]
for $H\neq H'$.

In case $H=H'$ and $P,P'\subset N\backslash H$, we have
\begin{gather*}
\langle\pi(a_{\mathbf{t}})e_{H}\wedge e_{P}\otimes e_{N\backslash P}\wedge e_{n+1}\wedge e_{n+2}\wedge\cdots\wedge e_{n+m}, e_{H}\wedge e_{P'}\otimes e_{N\backslash P'}\wedge e_{n+1}\wedge e_{n+2}\wedge\cdots\\
\qquad{}\wedge e_{n+m}\rangle
=\delta_{PP'}\cos t_{H}\cos^{2}t_{P}.
\end{gather*}
This gives
\[\left(
 \begin{matrix}
 n-s \\
 i \\
 \end{matrix}
\right)\frac{\big\langle\pi(a_{\mathbf{t}})j_W\big(e_{H}\otimes(e_{N})^{\otimes b}\big),j_W\big(e_{H'}\otimes(e_{N})^{\otimes b}\big)\big\rangle}{\big\langle j_W\big(e_{H}\otimes(e_{N})^{\otimes b}\big),j_W\big(e_{H'}\otimes(e_{N})^{\otimes b}\big)\big\rangle}
=\cos t_{H}\cos^{b} t_{N}\sum_{\substack{I\in N\backslash H\\|I|=i}} \cos^{2} t_{I},\]
which proves the lemma.
\end{proof}

\begin{Remark}\label{nueqmu}
For $i=0$, i.e., $\nu_0=\mu$,
we have the tensor product decomposition
\[V_{\omega_s}^G\otimes\big(V_{\omega_n}^G\big)^{\otimes b}\cong V_{\nu_0}^G\oplus\bigoplus_{\lambda'\precneqq\nu_0} m_{\lambda'}V_{\lambda'}^G,\]
and it leads to
\[Q^\mu_{\nu_{0}}(a_{\mathbf{t}})= d_{\nu_0}\Phi_{\nu_{0}}^{\mu}(a_{\mathbf{t}})+\sum_{\lambda'\precneqq\nu_0} d_{\lambda'}\Phi_{\lambda'}^{\mu}(a_{\mathbf{t}}),\qquad d_{\lambda'}\in\mathbb{C}.\]
Since the weight vector of $V_\mu^K$ is also the weight vector of $V_{\lambda'}^G\supset V_{\mu}^K$, we have
 $V_\mu^K\nsubseteq V_{\lambda'}^G$ for~$\lambda'\precneqq\nu_0$. So $d_{\lambda'}=0$ for~$\lambda'\precneqq\nu_0$. In this case, we have $\nu_0\in P_G^+(\mu)$ and $\Phi_{\nu_{0}}^{\mu}(a_{\mathbf{t}})=Q^\mu_{\nu_{0}}(a_{\mathbf{t}})$.
\end{Remark}
By considering the total degree, we see that $Q^\mu_{\nu_{i}}(a_{\mathbf{t}})$'s, $i=0,1,\dots,n-s$, are linearly independent.
\begin{Remark}\label{rem:weylgroupelementaction}
For the matrix spherical function $\Phi_{\lambda}^{\mu}(a_{\mathbf{t}})\in \operatorname{End}_{M}\big(V_{\mu}^{K}\big)$, we have
 \[\Phi_{\lambda}^{\mu}(a_{\mathbf{t}}) e_{1}\wedge e_{2}\wedge\cdots\wedge e_{s}\otimes (e_{\mathrm{N}})^{\otimes b}=P(\cos \mathbf{t})e_{1}\wedge e_{2}\wedge\cdots\wedge e_{s}\otimes (e_{\mathrm{N}})^{\otimes b},\]
where $P(\cos \mathbf{t})$ is a polynomial in $\cos t_i$'s.

For an $s$-tuple $H=\{h_{1},h_{2},\dots,h_{s}\}$ with $h_{i}<h_{i+1}$, the $(n-s)$-tuple $\bar{H}$ is the ordered tuple so that $\bar{H}=\{h_{s+1}, h_{s+2},\dots,h_{n}\}$ with $h_{i}<h_{i+1}$ and $H\cup\bar{H}=\mathrm{N}$. Let $e_{H}\otimes (e_{\mathrm{N}})^{\otimes b}\in V_{\mu}^{K}$ and $n_w\in N_{K'}(A')$ in Lemma \ref{lem:NKa} be a representative of
\[w=\left(
 \begin{matrix}
 1 & 2 & \cdots & s & s+1 & s+2 & \cdots & n \\
 h_{1} & h_{2} & \cdots & h_{s} & h_{s+1} & h_{s+2} & \cdots & h_{n} \\
 \end{matrix}
 \right)\]
such that
\[n_{w}^{-1}a_{(t_{1},t_{2},\dots,t_{n})}n_{w}=a_{(t_{h_{1}},t_{h_{2}},\dots,t_{h_{n}})}\]
and
\[\pi_\mu^K(n_{w})e_{1}\wedge e_{2}\wedge\cdots\wedge e_{s}\otimes (e_{\mathrm{N}})^{\otimes b}=e_{h_{1}}\wedge e_{h_{2}}\wedge\cdots\wedge e_{h_{s}}\otimes (e_{\mathrm{N}})^{\otimes b}.\]
It leads to
\begin{align*}
\Phi_{\lambda}^{\mu}(a_{\mathbf{t}})e_{h_{1}}\!\wedge e_{h_{2}}\!\wedge\!\cdots\!\wedge e_{h_{s}}\otimes\! \!(e_{\mathrm{N}})^{\otimes b}
\!&=\!\pi_\mu^K(n_{w})\pi_\mu^K\big(n_{w}^{-1}\big)\Phi_{\lambda}^{\mu}(a_{\mathbf{t}})\pi_\mu^K(n_{w})e_{1}\!\wedge e_{2}\!\wedge\!\cdots\!\wedge e_{s}\!\otimes\! \!(e_{\mathrm{N}})^{\otimes b}\!\\
&=\pi_\mu^K(n_{w})w(P)(\cos \mathbf{t})e_{1}\wedge e_{2}\wedge\cdots\wedge e_{s}\otimes (e_{\mathrm{N}})^{\otimes b}\\
&=w(P)(\cos \mathbf{t})e_{h_{1}}\wedge e_{h_{2}}\wedge\cdots\wedge e_{h_{s}}\otimes (e_{\mathrm{N}})^{\otimes b}.
\end{align*}
In this equation, $w(P)(\cos \mathbf{t})$ is the polynomial in $\cos t_{i}$'s where we let the Weyl group element~$w$ acts on $P(\cos \mathbf{t})$.

Since all the $M$-types in $V_\mu^K$, each $M$ type being $1$-dimensional and spanned by $e_H\otimes (e_\mathrm{N})^{\otimes b}$, are in a single Weyl group orbit for the reduced Weyl group, we only need to calculate the first entry of the corresponding spherical function and we can get the other entries by the action of the reduced Weyl group.
\end{Remark}

\begin{Lemma}\label{lem:rqi}
We have
\[(RQ^\mu_{\nu_{i}})(a_{\mathbf{t}})=c_{\nu_{i}}Q^\mu_{\nu_{i}}(a_{\mathbf{t}})
-2(n-s-i+1)(b+n-s-i+1)Q^\mu_{\nu_{i-1}}(a_{\mathbf{t}})\]
with
$c_{\nu_{i}}=\langle\nu_{i},\nu_{i}\rangle+2\langle\nu_{i},\rho\rangle$ and we define $Q_{\nu_{-1}}^{\mu}\equiv0$.
\end{Lemma}
\begin{proof}
The expression of $R$ is given in Appendix~\ref{app:rofcasimiroperator}. This lemma can be proved by using computer algebra, in particular \textsc{Maxima}, and some intermediate calculations are shown in Appendix~\ref{app:casimirop}.
\end{proof}

\begin{Remark}\label{eigenvalueofomegas}
We have $c_{\nu_{i}}>c_{\nu_{i-1}}$ by Lemma \ref{eigveldiff} since $\nu_{i}-\nu_{i-1}\succneqq0$.
\end{Remark}

\subsection[Example: mu=0]{Example: $\boldsymbol{\mu=0}$}
Now we calculate the zonal spherical function $\phi_i$ corresponding to $\omega_{i}+\omega_{m+n-i}$, $i=1,2,\dots,n$, by
\begin{align}
\phi_{i}\colon\ A&\rightarrow P(\cos\mathbf{t}),\nonumber\\
a_{\mathbf{t}}&\mapsto\frac{\langle\pi(a_{\mathbf{t}})v_{i},v_{i}\rangle}{\langle v_{i},v_{i}\rangle},\label{zonalsphericalfunction}
\end{align}
where $v_i$ is the $K$-fixed vector in $V_{\omega_i+\omega_{n+m-i}}^G$. Recalling Section \ref{sec:radialpartofthecasimiroperator}, $\phi_{i}$, $i=1,2,\dots,n$, is an eigenfunction of the radial part $R$ of the Casimir operator with eigenvalue
\begin{gather*}
d_{i}=\langle\omega_{i}+\omega_{m+n-i},\omega_{i}+\omega_{m+n-i}\rangle+2\langle\omega_{i}+\omega_{m+n-i},\rho\rangle
=2i(m+n-i+1).
\end{gather*}

Instead of calculating the zonal spherical function $\phi_{i}$, we calculate related bi-$K$-invariant function $\psi_{i}$ as matrix elements of $K$-fixed vector in a reducible $G$ representation. By calculating~$R(\psi_{i})$, we can relate them to the zonal spherical function as defined in \eqref{zonalsphericalfunction}.
\begin{Corollary}
Let
\[v_{i}'=\sum_{\substack{|P|=i\\P\subset \mathrm{N}}} (-1)^{b(P)}e_{P}\otimes e_{\mathrm{N}\backslash P}\wedge e_{\mathrm{M}}\in
 V_{\omega_{i}}^{G}\otimes V_{\omega_{m+n-i}}^{G},\]
where $i=0,1,\dots,n$, then $v'_{i}$ is a $K$-fixed vector in $V_{\omega_{i}}^{G}\otimes V_{\omega_{m+n-i}}^{G}$.
\end{Corollary}
\begin{proof}
It is a special case of Theorem~\ref{kintertwiner}.
\end{proof}

By Theorem \ref{gdec}, we have
\[V_{\omega_{i}}^{G}\otimes V_{\omega_{m+n-i}}^{G}\cong\bigoplus_{j=0}^{i} V_{\omega_{j}+\omega_{m+n-j}}^{G},
\qquad j=0,1,\dots,i,\]
then $v'_{i}$ is a linear combination of $K$-fixed vector in $V_{\omega_{j}+\omega_{m+n-j}}^{G}$ with $j=0,1,\dots,i$. We define
\[\psi_{i}(a_{\mathbf{t}})=\langle\pi(a_{\mathbf{t}})v_{i}',v_{i}'\rangle=\sum_{\substack{|P|=i\\P\subset \mathrm{N}}} \cos^{2}t_{P}\]
and $\psi_{i}$ is $i$-th elementary symmetric polynomial in $\cos^{2}t_{k}$, $k=1,2,\dots,n$. We use the convention $\phi_{0}=\psi_{0}=1$, and $d_{0}=0$. In this case, $\psi_{i}$'s are linearly independent since the total degree of the $\psi_i$'s as polynomials in $(\cos t_{1},\cos t_2,\dots,\cos t_n)$ are all different.
\begin{Corollary}\label{cor:zonalsf}
We have
\[(R\psi_{i})(\cos\mathbf{t})=d_{i}\psi_{i}(\cos\mathbf{t})-2(n-i+1)^{2}\psi_{i-1}(\cos\mathbf{t}),\]
where we let $\psi_{-1}(\cos\mathbf{t})=0$.
\end{Corollary}
\begin{proof}
We can prove it by Lemma \ref{lem:rqi}, where we let $s=b=0$.
\end{proof}
\begin{Proposition}\label{prop:leadingtermofphii}
In this case, $\phi_{i}$ can be written as linear combination of $\psi_{j}$'s with $j=0,1,\dots,i$, and the coefficient of $\psi_{i}$ is non-zero. Explicitly, we have
\[\phi_i(a_\mathbf{t})=l\sum_{j=0}^{i} k_j\psi_j(\cos\mathbf{t}),\]
where
\[k_j=\frac{(-1)^{j}(i+1-j)_{j}(m+n+2-i-j)_{j}}{(n+1-j)_j(n+1-j)_{j}}\]
and
\[l=(-1)^i\frac{(-n)_i}{(-m)_i}.\]
\end{Proposition}
\begin{proof}
It is true for $i=0$ and $\phi_0(a_\mathbf{t})=\psi_0(\cos\mathbf{t})=1$. By Corollary \ref{cor:zonalsf} and $(R\phi_i)(a_\mathbf{t})=d_i\phi_i(a_\mathbf{t})$, we have
\[-2(n-j)^2k_{j+1}+d_j k_j=d_i k_j,\]
then let $k_0=1$ and the expression of $k_j$'s are clear. Also let $\mathbf{t}=(0,0,\dots,0)$, and we have
\[\phi_i(a_\mathbf{t})=1=l\sum_{j=0}^{i}\frac{(-1)^{j}(i+1-j)_{j}(m+n+2-i-j)_{j}}{j! (n+1-j)_{j}}.\]
It leads to
\[l=\left(\rFs{2}{1}{-i,i-m-n-1}{-n}{1}\right)^{-1}=(-1)^i\frac{(-n)_i}{(-m)_i}.\]
by Chu--Vandermonde summation. Then this proposition is proved by calculation.
\end{proof}

\section[Matrix-valued spherical functions with mu=a omega\_1+b omega\_n]{Matrix-valued spherical functions with $\boldsymbol{\mu=a\omega_{1}+b\omega_{n}}$}\label{exa:aomega1+bomegan}
The goal of this section is to give the approximation of the matrix spherical function corresponding to $\mu=a\omega_1+b\omega_n$ with $a,b\in\mathbb{N}$. In this case, $B(\mu)$ is explicit in Section~\ref{caseofaomega1bomegan}. We have the tensor product decomposition, recall $\sum_{i=1}^{n}a_i=a$, $a_i\in\mathbb{N}$,
\[W_{\nu}^{G}=\bigotimes_{i=1}^n \big(V_{\omega_{i}}^{G}\otimes V_{\omega_{m+n+1-i}}^{G}\big)^{\otimes a_{i}}\otimes \big(V_{\omega_{n}}^{G}\big)^{\otimes b}\cong\bigoplus_{\lambda\preccurlyeq\nu} m_{\lambda}V_{\lambda}^{G},\qquad m_\nu=1.\]
We recall the notation in \eqref{weightvectoreta} and define
\[w_{i}=\sum_{\substack{|P|=i-1\\P\subset \mathrm{N}}} (-1)^{b(P)}e_{1}\wedge e_{P}\otimes e_{\mathrm{N}\backslash P}\wedge e_\mathrm{M}\in V_{\omega_{i}}^{G}\otimes V_{\omega_{m+n+1-i}}^{G}.\]
So
\begin{Lemma}\label{lem:aomega1bomeganhighestweightvector}
\[v_{\mu}=\bigotimes_{i=1}^n  w_{i}^{\otimes a_{i}}\otimes (e_\mathrm{N})^{\otimes b}\in W_{\nu}^{G}\]
is a $K$-highest weight vector of weight $\mu=a\omega_{1}+b\omega_{n}$.
\end{Lemma}
\begin{proof}
It can be proved by using the fact that $E_{i,i+1}v_\mu=0$ for $i=1,2,\dots,n-1,n+1,\allowbreak\dots,n+m$.
\end{proof}
\begin{Remark}
We can calculate other weight vectors in $V_\mu^K$ by the Chevalley basis acting on~$v_\mu$. Then we have a $K$-intertwiner $j_W$ from $V_{\mu}^{K}$ to the $G$-module $W_{\nu}^{G}$. We have $Q_{\nu}^{\mu} \colon A \to \operatorname{End}\big(V^K_\mu\big),
\quad a\mapsto j_W^\ast \circ \pi_{W_{\nu}^{G}}(a) \circ j_W$ which is the corresponding matrix spherical function restricted to $A$. Since $V_{\mu}^{K}|_{M}$ splits multiplicity free and $Q_{\nu}^{\mu}(a)\in\operatorname{End}_{M}\big(V_{\mu}^{K}\big)$, $Q_{\nu}^{\mu}(a)$ is a diagonal matrix if we choose the $M$-weight vectors as the basis of $V_{\mu}^{K}$. We calculate the entry corresponding to $v_{\mu}$ and
\[\frac{\langle\pi(a_{\mathbf{t}})v_{\mu},v_{\mu}\rangle}{\langle v_{\mu},v_{\mu}\rangle}
=\frac{1}{\lVert v_{\mu} \rVert ^2}\cos^{a} t_{1}
\times\big(\psi_{1}^{(1)}(\cos\mathbf{t})\big)^{a_{2}}\big(\psi_{2}^{(1)}(\cos\mathbf{t})\big)^{a_{3}}\cdots\big(\psi_{n-1}^{(1)}(\cos\mathbf{t})\big)^{a_{n}}\cos^{b} t_\mathrm{N},\]
where $\psi_i^{(1)}$ is $i$-th symmetric polynomial in $\cos^2 t_2,\cos^2 t_3,\dots,\cos^2 t_n$ as defined in Appendix~\ref{app:casimirop}. Then $Q_{\nu}^{\mu}(a_{\mathbf{t}})$'s are linearly independent for $\nu\in B(\mu)$. Other entries can be calculated analogously.
\end{Remark}
\begin{Remark}\label{nueqmuaomega1bomegan}
Similar to Remark \ref{nueqmu}, we have $\Phi_{a\omega_1+b\omega_n}^{\mu}(a_{\mathbf{t}})=Q^\mu_{a\omega_1+b\omega_n}(a_{\mathbf{t}})$.
\end{Remark}
\begin{Remark}\label{rem:Bmu}
For $\nu=\sum_{i=1}^{n}a_i(\omega_{i}+\omega_{m+n+1-i})\in B(\mu)$, assume there exists $\lambda'=\nu'+\lambda_{\rm sph}=\sum_{i=1}^{n}a'_i(\omega_{i}+\omega_{m+n+1-i})+\sum_{i=1}^{n}d_i(\omega_{i}+\omega_{m+n-i})\in P_G^+(\mu)$ such that $\nu\succcurlyeq\lambda'$. We want to show $\lambda_{\rm sph}=0$. Then
plugging the expression of $\nu$, $\lambda'$, we get
\[\nu-\lambda'=\sum_{i=1}^{n}(a_{i}-a_{i}')(\omega_{i}+\omega_{m+n+1-i})-\sum_{i=1}^n d_i(\omega_{i}+\omega_{m+n-i}),\]
where $\omega_{0}=\omega_{m+n}=0$ by convention.

By summation by parts, we get
\begin{gather*}
\nu-\nu'=\sum_{i=1}^{n-1}
\biggl(\sum_{j=1}^{i} (a_{j}'-a_{j})\biggr)(\omega_{i+1}+\omega_{m+n-i}-\omega_{i}-\omega_{m+n+1-i})
+\sum_{j=1}^{n}(a_{j}-a_{j}')(\omega_{n}+\omega_{m+1}).\!\!\!
\end{gather*}

First, observe that
\[\sum_{j=1}^{n}(a_{j}-a_{j}')=a-a=0\]
since $\nu,\nu'\in B(\mu)$.

Next note that
\[\omega_{i+1}+\omega_{m+n-i}-\omega_{i}-\omega_{m+n+1-i}=\epsilon_{i+1}-\epsilon_{m+n+1-i}=\sum_{s=i+1}^{m+n-i} \alpha_{s},\]
by Remark~\ref{weightrootrelation}. Then the coefficient of $\alpha_1$ in $\nu-\lambda'$ is $-\sum_{i=1}^{n}d_i$. Since $\nu-\lambda'\in Q^+$ and $d_i\in\mathbb{N}$, we have $d_i=0$ for $i=1,2,\dots,n$. It follows that $\lambda'=\nu'\in B(\mu)$.
\end{Remark}

\begin{Lemma}\label{lem:phiaomega1+bomegan}
We have
\[\Phi_{\nu}^{\mu}(a_{\mathbf{t}})=\sum_{\lambda\preccurlyeq\nu,\,\lambda\in B(\mu)} d_{\lambda}Q_{\lambda}^{\mu}(a_{\mathbf{t}}),\qquad d_\lambda\in\mathbb{C}, \]
and $d_{\nu}\neq0$.
\end{Lemma}
\begin{proof}
It is true for $\nu=a\omega_{1}+b\omega_{n}$ since $\Phi_{a\omega_{1}+b\omega_{n}}^{\mu}(a_{\mathbf{t}})=Q_{a\omega_{1}+b\omega_{n}}^{\mu}(a_{\mathbf{t}})$ by Remark~\ref{nueqmuaomega1bomegan}. We assume it is true for all the elements occurring in the subset $\{ \widetilde{\nu}\precneqq\nu\mid\widetilde{\nu}\in B(\mu)\}\subset B(\mu)$ and~$\Phi_{\widetilde{\nu}}^{\mu}(a_{\mathbf{t}})$ can also be written as linear combination of $Q_{\widetilde{\nu}'}^{\mu}(a_\mathbf{t})$'s with $\widetilde{\nu'}\preccurlyeq\widetilde{\nu}$. For $\nu\in B(\mu)$, we have
\[Q_{\nu}^{\mu}(a_{\mathbf{t}})=\sum_{\lambda\preccurlyeq\nu} d_{\lambda}\Phi_{\lambda}^{\mu}(a_{\mathbf{t}})\]
from the tensor product decomposition
 \[W_\nu^G\cong\bigoplus_{\lambda\preccurlyeq\nu} m_{\lambda}V_{\lambda}^{G}.\]
If $d_{\nu}=0$, by Remark~\ref{rem:Bmu} we have
\[Q_{\nu}^{\mu}(a_{\mathbf{t}})=\sum_{\lambda\precneqq\nu,\, \lambda\in B(\mu)} d_{\lambda}\Phi_{\lambda}(a_{\mathbf{t}})
=\sum_{\lambda\precneqq\nu,\, \lambda\in B(\mu)} d'_{\lambda}Q_{\lambda}^{\mu}(a_{\mathfrak{t}})\]
and it contradicts the fact that $Q_{\lambda}^{\mu}$'s are linearly independent. So this lemma follows.
\end{proof}
\begin{Theorem}\label{thm:muaomega1bomegan}
For $\lambda=\nu+\sum_{i=1}^{n}d_i(\omega_i+\omega_{n+m-i})$, where $\nu\in B(\mu)$ and $d_i\in\mathbb{N}$, we define
 \[Q_\lambda^\mu(a_\mathbf{t})=\psi_1^{d_1}\psi_2^{d_2}\cdots\psi_n^{d_n}Q_\nu^\mu(a_{\mathbf{t}}),\]
then we have
\[\Phi_{\lambda}^{\mu}(a_{\mathbf{t}})=\sum_{\lambda'\preccurlyeq\lambda,\, \lambda'\in P_G^+(\mu)} d_{\lambda'}Q_{\lambda'}^{\mu}(a_{\mathbf{t}}),\qquad d_{\lambda'}\in\mathbb{C}, \]
and $d_{\lambda}\neq0$.
\end{Theorem}
\begin{proof}
 It can be proved by using Lemma \ref{lem:sphericalpolynomials}, Proposition \ref{prop:leadingtermofphii}, and Lemma \ref{lem:phiaomega1+bomegan}.
\end{proof}

\begin{Example}
For $\mu=\omega_{1}+b\omega_{n}$ with $b\in\mathbb{N}$, using Proposition \ref{prop:P+Gmuofrankn}, we have
\[B(\mu)=\{\nu_{i}=\omega_{1+i}+\omega_{m+n-i}+b\omega_{n}\mid i=0,1,\dots,n-1\},\]
and we have the corresponding approximate function $Q^\mu_{\nu_{i}}(a_{\mathbf{t}})$ by \eqref{approfuncQi}. Using Lemma \ref{lem:rqi}, we have
\[(RQ^\mu_{\nu_{i}})(a_{\mathbf{t}})=c_{\nu_{i}}Q^\mu_{\nu_{i}}(a_{\mathbf{t}})
-2(n-i)(b+n-i)Q^\mu_{\nu_{i-1}}(a_{\mathbf{t}}).\]
Then
\[\Phi_{\nu_i}^\mu(a_\mathbf{t})=l\sum_{j=0}^{i}k_j Q_{\nu_j}^\mu(a_\mathbf{t}),\]
where
\[k_j=\frac{(-1)^{i-j}(n-i)_{i-j}(n+b-i)_{i-j}}{(i-j)!(m+n+b-2i+1)_{i-j}}\]
and
\[l=(-1)^i\left(
 \begin{matrix}
 n-1 \\
 i \\
 \end{matrix}
\right)^{-1}\frac{(m+n+b-2i+1)_i}{(-m)_i}\]
since $\Phi_{\nu_i}^\mu(a_\mathbf{t})$ is the eigenfunction of the radial part $R$ of the Casimir operator with eigenvalue~$c_{\nu_i}$. The calculation is similar to the proof of Proposition \ref{prop:leadingtermofphii}.
\end{Example}

\section[Matrix-valued spherical functions with mu=omega\_s+b omega\_n]{Matrix-valued spherical functions with $\boldsymbol{\mu=\omega_s+b\omega_n}$}\label{sec:omegasbomegan}
The goal of this section is to calculate $P_G^+(\omega_s+b\omega_n)$ with $s=0,1,\dots,n$ and $b\in\mathbb{N}$, and approximate the corresponding spherical functions. We recall some notations and results in Section~\ref{subsec:omegas} and Section \ref{geninj}.

\begin{Remark}
Note that, using the notation of \eqref{expressionoflambdaH} and Theorem \ref{gdec},
\[\lambda_H=\omega_{x_1}+\omega_{m+n-y_1}+\omega_{x_2}+\omega_{m+n-y_2}+\cdots+\omega_{x_o}+\omega_{m+n-y_o}+b\omega_n
 \succcurlyeq\sum_{i=1}^{o} \omega_{x_i-y_i}+b\omega_n.\]
Also we have
\[\sum_{i=1}^{o} \omega_{x_i-y_i}+b\omega_n\succcurlyeq\omega_{x_1-y_1+x_2-y_2}+\sum_{i=3}^{o} \omega_{x_i-y_i}+b\omega_n\succcurlyeq\cdots\succcurlyeq\omega_{\sum_{i=1}^{o} (x_i-y_i)}+b\omega_n
 =\omega_s+b\omega_n,\]
where we use Theorem \ref{gdec} iteratively. Then $\lambda_H\succcurlyeq\mu$.
\end{Remark}

\begin{Lemma}
Define $u$ as a linear map from $V_\mu^K$ to $V_{\omega_{x_1-y_1}}^K\otimes V_{\omega_{x_2-y_2}}^K\otimes\cdots\otimes V_{\omega_{x_o-y_o}}^K\otimes \big(V_{\omega_n}^K\big)^{\otimes b}$, where
\begin{gather*}
 u\big(e_{h_1}\wedge e_{h_2}\wedge\cdots\wedge e_{h_s}\otimes(e_\mathrm{N})^{\otimes b}\big)=\sum_{\sigma\in S_s} (-1)^{l(\sigma)}e_{h_{\sigma(1)}}\wedge e_{h_{\sigma(2)}}\wedge\cdots\wedge e_{h_{\sigma(x_1-y_1)}}\\
 \quad\otimes e_{h_{\sigma(x_1-y_1+1)}}\wedge e_{h_{\sigma(x_1-y_1+2)}}\wedge\cdots\wedge e_{h_{\sigma(x_1-y_1+x_2-y_2)}}\\
\quad \otimes\cdots\otimes e_{h_{\sigma(\sum_{i=1}^{o-1} (x_i-y_i)+1)}}\wedge e_{h_{\sigma(\sum_{i=1}^{o-1} (x_i-y_i)+2)}}\wedge\cdots\wedge e_{h_{\sigma(\sum_{i=1}^{o-1} (x_i-y_i)+x_o-y_o)}}\otimes(e_\mathrm{N})^{\otimes b}
\end{gather*}
with $H=\{1\leq h_1<h_2<\cdots<h_s\leq n\}$, then $u$ is a $K$-intertwiner.
\end{Lemma}
\begin{proof}
Let $E_{ij}$ be $(n+m)\times (n+m)$-matrix with one non-zero entry $1$ at $(i,j)$-entry. Then
\[\{E_{i,i+1}, E_{i+1,i},E_{ii}-E_{i+1,i+1}\mid i=1,2,\dots,n-1,n+1,n+2,\dots,n+m\}\]
can be considered as the Chevalley basis of the complex Lie algebra of $K$. So we need to prove~$u$ acting on $V_\mu^K$ commutes with the Chevalley basis action.

If $i,i+1\notin H$, then
\begin{align*}
&E_{i,i+1}u\big(e_{h_1}\wedge e_{h_2}\wedge\cdots\wedge e_{h_s}\otimes(e_\mathrm{N})^{\otimes b} \big)=u E_{i,i+1}\big(e_{h_1}\wedge e_{h_2}\wedge\cdots\wedge e_{h_s}\otimes(e_\mathrm{N})^{\otimes b} \big)=0,\\
&E_{i+1,i}u\big(e_{h_1}\wedge e_{h_2}\wedge\cdots\wedge e_{h_s}\otimes(e_\mathrm{N})^{\otimes b} \big)=u E_{i+1,i}\big(e_{h_1}\wedge e_{h_2}\wedge\cdots\wedge e_{h_s}\otimes(e_\mathrm{N})^{\otimes b} \big)=0,\\
&(E_{ii}-E_{i+1,i+1})u\big(e_{h_1}\wedge e_{h_2}\wedge\cdots\wedge e_{h_s}\otimes(e_\mathrm{N})^{\otimes b} \big)\\
&\qquad =u (E_{ii}-E_{i+1,i+1})\big(e_{h_1}\wedge e_{h_2}\wedge\cdots \wedge e_{h_s}\otimes(e_\mathrm{N})^{\otimes b} \big)=0.
 \end{align*}

If $i,i+1\in H$, then
\begin{align*}
 &u E_{i,i+1}\big(e_{h_1}\wedge e_{h_2}\wedge\cdots\wedge e_{h_s}\otimes(e_\mathrm{N})^{\otimes b} \big)=0,\\
 &u E_{i+1,i}\big(e_{h_1}\wedge e_{h_2}\wedge\cdots\wedge e_{h_s}\otimes(e_\mathrm{N})^{\otimes b} \big)=0,\\
 &u (E_{ii}-E_{i+1,i+1})\big(e_{h_1}\wedge e_{h_2}\wedge\cdots\wedge e_{h_s}\otimes(e_\mathrm{N})^{\otimes b} \big)\\
 &\qquad{} =(E_{ii}-E_{i+1,i+1})u\big(e_{h_1}\wedge e_{h_2}\wedge\cdots \wedge e_{h_s}\otimes(e_\mathrm{N})^{\otimes b} \big)=0.
 \end{align*}
Let $e_{X_1}\otimes e_{X_2}\otimes\cdots\otimes e_{X_o}\otimes(e_\mathrm{N})^{\otimes b}$ be a monomial in $u\big(e_{h_1}\wedge e_{h_2}\wedge\cdots\wedge e_{h_s} \otimes(e_\mathrm{N})^{\otimes b}\big)$ with permutation $\sigma$, then we have two possibilities.
\begin{itemize}\itemsep=0pt
 \item If $i,i+1\in X_k$, then
 \[E_{i,i+1}e_{X_1}\otimes e_{X_2}\otimes\cdots\otimes e_{X_o}\otimes(e_\mathrm{N})^{\otimes b}=0.\]
 \item If $i\in X_k$, $i+1\in X_l$ and $k\neq l$, then there is another monomial $e_{X'_1}\otimes e_{X'_2}\otimes\cdots\otimes e_{X'_o}\otimes(e_\mathrm{N})^{\otimes b}$ in $u\big(e_{h_1}\wedge e_{h_2}\wedge\cdots\wedge e_{h_s}\otimes(e_\mathrm{N})^{\otimes b} \big)$ with permutation $\sigma'$ where we flip the order of $i$ and $i+1$ in $X'_j$'s and $l(\sigma')=l(\sigma)+1$. So $E_{i,i+1}e_{X_1}\otimes e_{X_2}\otimes\cdots\otimes e_{X_o}\otimes(e_\mathrm{N})^{\otimes b}=E_{i,i+1}e_{X'_1}\otimes e_{X'_2}\otimes\cdots\otimes e_{X'_o}\otimes(e_\mathrm{N})^{\otimes b}$ is a monomial where we change $i+1$ to $i$ in $X_l$ and~$X'_k$. Then
\[E_{i,i+1}u\big(e_{h_1}\wedge e_{h_2}\wedge\cdots\wedge e_{h_s}\otimes(e_\mathrm{N})^{\otimes b}\big)=0.\]
Similarly, we have
\[E_{i+1,i}u\big(e_{h_1}\wedge e_{h_2}\wedge\cdots\wedge e_{h_s}\otimes(e_\mathrm{N})^{\otimes b}\big)=0.\]
\end{itemize}
If $i\notin H$, $i+1\in H$, then we have
\[u E_{i,i+1}\big(e_{h_1}\wedge e_{h_2}\wedge\cdots\wedge e_{h_s}\otimes(e_\mathrm{N})^{\otimes b}\big)=E_{i,i+1}u \big(e_{h_1}\wedge e_{h_2}\wedge\cdots\wedge e_{h_s}\otimes(e_\mathrm{N})^{\otimes b}\big),\]
where we change $i+1$ to $i$ in $H=\{h_1,h_2,\dots,h_s\}$. Also we have
\[u E_{i+1,i}\big(e_{h_1}\wedge e_{h_2}\wedge\cdots\wedge e_{h_s}\otimes(e_\mathrm{N})^{\otimes b}\big)=E_{i+1,i}u \big(e_{h_1}\wedge e_{h_2}\wedge\cdots\wedge e_{h_s}\otimes(e_\mathrm{N})^{\otimes b}\big)=0,\]
and
\begin{align*}
 &u(E_{ii}-E_{i+1,i+1})\big(e_{h_1}\wedge e_{h_2}\wedge\cdots\wedge e_{h_s}\otimes(e_\mathrm{N})^{\otimes b}\big)=(E_{ii}-E_{i+1,i+1})u\big(e_{h_1}\wedge e_{h_2}\wedge\cdots\\
 &\quad\wedge e_{h_s}\otimes(e_\mathrm{N})^{\otimes b}\big)=-(b+1)u\big(e_{h_1}\wedge e_{h_2}\wedge\cdots\wedge e_{h_s}\otimes(e_\mathrm{N})^{\otimes b}\big).
\end{align*}
The situation of $i\in H$, $i+1\notin H$ is similar.
\end{proof}

We define a $K$-intertwiner $w$ from $u\big(V_{\omega_s+b\omega_n}^K\big)$ to $W_{\lambda}^G$ using Theorem \ref{kintertwiner} such that
\begin{align*}
 &w u\big(e_{h_1}\wedge e_{h_2}\wedge\cdots\wedge e_{h_{s}}\otimes(e_\mathrm{N})^{\otimes b}\big)=\sum_{\sigma\in S_s} (-1)^{l(\sigma)}
 \Bigg(\sum_{\substack{P_1\subset \mathrm{N}\\|P_1|=y_1}} (-1)^{b(P_1)}e_{h_{\sigma(1)}}\wedge e_{h_{\sigma(2)}}\wedge\cdots\\&
 \wedge e_{h_{\sigma(x_1-y_1)}}\wedge e_{P_1}\otimes e_{\mathrm{N}\backslash P_1}\wedge e_{\mathrm{M}}\Bigg)\otimes \Bigg(\sum_{\substack{P_2\subset \mathrm{N}\\|P_2|=y_2}} (-1)^{b(P_2)}e_{h_{\sigma(x_1-y_1+1)}}\wedge e_{h_{\sigma(x_1-y_1+2)}}\wedge\cdots\\
 &\wedge e_{h_{\sigma(x_1-y_1+x_2-y_2)}}\wedge e_{P_2}\otimes e_{\mathrm{N}\backslash P_1}\wedge e_{\mathrm{M}}\Bigg)
 \otimes\cdots\otimes \Bigg(\sum_{\substack{P_o\subset \mathrm{N}\\|P_o|=y_o}} (-1)^{b(P_o)}e_{h_{\sigma(\sum_{i=1}^{o-1} (x_i-y_i)+1)}}\\
 &\wedge e_{h_{\sigma(\sum_{i=1}^{o-1} (x_i-y_i)+2)}}\! \wedge\cdots\wedge e_{h_{\sigma(\sum_{i=1}^{o-1} (x_i-y_i)+x_o-y_o)}}\!\wedge e_{P_o}\otimes e_{\mathrm{N}\backslash P_o}\wedge e_{\mathrm{M}}\Bigg)\!\otimes(e_\mathrm{N})^{\otimes b}
 \otimes \bigotimes_{i=1}^n (v'_i)^{\otimes d_i},
\end{align*}
where $v'_i$ is $K$-fixed vector in $V_{\omega_i}^G\otimes V_{\omega_{n+m-i}}^G$, see Section \ref{geninj}. Let $j_W=w u$, then we can calculate the approximate function $Q_{\lambda}^\mu$. We let
\[Q^\mu_{\lambda} \colon \ A \to \operatorname{End}\big(V^K_\mu\big),
\qquad
a_\mathbf{t}\mapsto
 p\cdot j_W^\ast \circ \pi_{W_{\lambda}^{G}}(a_\mathbf{t}) \circ j_W,\]
where
\[p=\prod_{j=1}^{o} \left(
 \begin{matrix}
 n-x_j+y_j \\
 y_j \\
 \end{matrix}
\right)\cdot \prod_{i=1}^n \left(
 \begin{matrix}
 n \\
 i \\
 \end{matrix}
\right)^{d_i},
\]
and we choose $e_{H}\otimes(e_\mathrm{N})^{\otimes b}$($H=\{1\leq h_1<h_2<\cdots<h_s\leq n\}$) as basis of $V_\mu^K$. The entry of~$Q_{\nu_i}^{\mu}(a_{\mathbf{t}})$ corresponding to $e_1\wedge e_2\wedge\cdots\wedge e_s\otimes(e_\mathrm{N})^{\otimes b}$ is
\begin{gather*}
 f=\cos t_1\cos t_2\cdots \cos t_s(\cos t_1\cos t_2\cdots \cos t_n)^b\psi_1^{d_1}\psi_2^{d_2}\cdots\psi_n^{d_n}\\
 \hphantom{f=}{}\times\sum_{\sigma\in S_s} \psi_{y_1}^{\{\sigma(1),\sigma(2),\dots,\sigma(x_1-y_1)\} }
 \psi_{y_2}^{\{\sigma(x_1-y_1+1),\sigma(x_1-y_1+2),\dots,\sigma(x_1-y_1+x_2-y_2)\} }\cdots\\
 \hphantom{f=}{}\times\psi_{y_o}^{\{\sigma(\sum_{i=1}^{o-1} (x_i-y_i)+1),
 \sigma(\sum_{i=1}^{o-1}
 (x_i-y_i)+2),\dots,\sigma(\sum_{i=1}^{o-1} (x_i-y_i)+x_o-y_o)\}},
\end{gather*}
where the notations are given by Appendix \ref{app:casimirop}. Other entries can be calculated using Remark~\ref{rem:weylgroupelementaction}.
\begin{Remark}\label{rem:polynomialspartialordering}
We define the lattice of $\omega_1,\omega_2,\dots,\omega_n,\omega_m,\omega_{m+1},\dots,\omega_{n+m-1}$ as
\[\overline{P}=\Bigg\{\sum_{i=1}^n a_i\omega_i+\sum_{i=m}^{n+m-1} a_i\omega_i\mid a_i\in\mathbb{Z}\Bigg\}.\]
We define
\[\omega_i|_A=\omega_{n+m-i}|_A=\mathrm{i}\sum_{j=1}^{i}t_j,\qquad i=1,2,\dots,n.\]
Then for $\sum_{i=1}^{n}a_i\omega_i+\sum_{i=m}^{n+m-1}a_i\omega_i=\eta\in\overline{P}$, we have
\[\eta|_A=\sum_{i=1}^n (a_i+a_{n+m-i})\mathrm{i}\sum_{j=1}^{i}t_j.\]
If there is another weight $\eta'\in\overline{P}$ such that $\eta'|_A=\eta|_A$, then we have
\[\eta-\eta'=\sum_{i=1}^n b_i(\omega_i-\omega_{n+m-i}),\qquad b_i\in\mathbb{Z}.\]
Using the result in \cite[p.~250]{Bour}, where
\begin{gather*}
 \omega_i=\frac{1}{n+m}\big[(n+m-i)\alpha_1+2(n+m-i)\alpha_2+\cdots+(i-1)(n+m-i)\alpha_{i-1}\\
\hphantom{\omega_i=\frac{1}{n+m}\big[}{} +i(n+m-i)\alpha_i+i(n+m-1-i)\alpha_{i+1}+\cdots+i\alpha_{n+m-1}\big],
\end{gather*}
then it leads to $\eta-\eta'\notin Q^+\cup(-Q^+)$. We let $f_H$ be the entry corresponding to $e_H\otimes (e_\mathrm{N})^{\otimes b}$ in~$Q_{\lambda_H}^\mu(a_\mathbf{t})$ and we have
\[f_H=k_{\lambda_H}{\rm e}^{\lambda_H|_A}+\sum_{\lambda'\precneqq\lambda_H} k_{\lambda'}{\rm e}^{\lambda'|_A},\qquad k_{\lambda_H}\neq0.\]
Similarly, $\psi_i(\cos \mathbf{t})$ can be written as
\[\psi_i(\cos \mathbf{t})=k_{\lambda_i}{\rm e}^{\lambda_i|_A}+\sum_{\tau\precneqq\lambda_i} k_{\tau}{\rm e}^{\tau|_A},\qquad k_{\lambda_i}\neq0.\]
\end{Remark}
\begin{Proposition}\label{prop:B'mu}
We have
\[\Phi_{\lambda_H}^\mu(a_\mathbf{t})=\sum_{\lambda'\preccurlyeq\lambda_H} d_{\lambda'}Q_{\lambda'}^\mu(a_\mathbf{t}),\qquad d_{\lambda_H}\neq0.\]
Moreover, we have $\lambda_H\in P_G^+(\mu)$.
\end{Proposition}
\begin{proof}
It is true for $\lambda_H=\mu$ since $Q_\mu^\mu(a_\mathbf{t})=\Phi_\mu^\mu(a_\mathbf{t})$ by Remark \ref{nueqmu}, and we assume it is true for $\lambda_H\precneqq\nu\in P_G^+$. For $\lambda_H=\nu$, $Q_{\lambda_H}^\mu(a_\mathbf{t})$ is the linear combination of $\Phi_{\lambda'}^{\mu}(a_\mathbf{t})$ where $\lambda'\preccurlyeq\lambda_H$. Note that $\Phi_{\lambda'}^\mu(a_\mathbf{t})\equiv0$ if $\lambda'\notin P_G^+(\mu)$ by definition. Explicitly,
\[Q_{\lambda_H}^\mu(a_\mathbf{t})=p_{\lambda_H}\Phi_{\lambda_H}^\mu(a_\mathbf{t})+\sum_{\lambda'\precneqq\lambda_H} p_{\lambda'}\Phi_{\lambda'}^\mu(a_\mathbf{t})=p_{\lambda_H}\Phi_{\lambda_H}^\mu(a_\mathbf{t})+\sum_{\lambda'\precneqq\lambda_H} \bar{p}_{\lambda'}Q_{\lambda'}^\mu(a_\mathbf{t}).\]
For any $\lambda'\in P_G^+(\mu)$, we have
\[\lambda'=\lambda_{H'}+\lambda'_{\rm sph}=\lambda_{H'}+\sum_{i=1}^n d_i(\omega_i+\omega_{n+m-i}),\]
see Lemma \ref{lem:PG+mudec}.

Now we prove $p_{\lambda_H}\neq0$. For the entry of $Q_{\lambda_H}^\mu(a_\mathbf{t})$ corresponding to $e_H\otimes (e_\mathrm{N})^{\otimes b}$, we have
\[f_H=k_{\lambda_H}{\rm e}^{\lambda_H|_A}+\sum_{\lambda'\precneqq\lambda_H} k_{\lambda'}{\rm e}^{\lambda'|_A},\qquad k_{\lambda_H}\neq0.\]
For other $Q_{\lambda'}^\mu(a_\mathbf{t})$, where $P_G^+(\mu)\ni\lambda_{H'}+\sum_{i=1}^{n}d_i(\omega_i+\omega_{n+m-i})=\lambda'\precneqq\lambda_H$, the corresponding entry is
\[\prod_{i=1}^n \bigg(k_{\lambda_i}{\rm e}^{\lambda_i|_A}+\sum_{\tau\precneqq\lambda_i} k_{\tau}{\rm e}^{\tau|_A}\bigg)^{d_i}
\bigg(k_{\lambda_{H'}}{\rm e}^{w(\lambda_{H'})|_A}+\sum_{\xi\precneqq\lambda_{H'}} k_\xi {\rm e}^{w(\xi)|_A}\bigg),\]
and $\sum_{i=1}^{n}d_i(\omega_i+\omega_{n+m-i})+w(\lambda_{H'})\preccurlyeq\lambda'\precneqq\lambda_H$, where $\sum_{i=1}^{n}d_i(\omega_i+\omega_{n+m-i})+w(\lambda_{H'})$, $\lambda',\lambda_H\in\overline{P}$.
If $p_{\lambda_H}=0$, then there exists a weight $\eta\precneqq\lambda_H$ such that $\eta|_A=\lambda_H|_A$ which contradicts the result mentioned in Remark \ref{rem:polynomialspartialordering}. So this proposition is proved.
\end{proof}

Similarly, we have
\begin{Theorem}\label{thm:omegasbomegan}
Let $\lambda\in B'(\mu)+P_G^+(0)$, then we have
\begin{equation}\label{lambdaforaomega1+bomega2}
 \Phi_{\lambda}^\mu(a_\mathbf{t})=\sum_{\lambda'\preccurlyeq\lambda} d_{\lambda'}Q_{\lambda'}^\mu(a_\mathbf{t}),\qquad d_{\lambda}\neq0
 \end{equation}
and
\[P_G^+(\mu)=B'(\mu)+P_G^+(0).\]
\end{Theorem}
\begin{proof}
By Lemma \ref{lem:PG+mudec}, we have $P_G^+(\mu)\subset B'(\mu)+P_G^+(0)$. The proof of \eqref{lambdaforaomega1+bomega2} is similar to Proposition \ref{prop:B'mu}, and then it leads to $P_G^+(\mu)\supset B'(\mu)+P_G^+(0)$. So this theorem is proved.
\end{proof}
\begin{Remark}
In van Pruijssen \cite{Prui-IMRN}, some examples for other symmetric pairs are given.
\end{Remark}

\section{Matrix-valued orthogonal polynomials}\label{sec:matrixweight}
The goal of this section is to give the matrix weight for the case of $\mu=a\omega_1+b\omega_n$ and $\mu=\omega_s+b\omega_n$. Note that for these two cases, we have $\sharp B(\mu)=\dim V_\mu^K$.

We put the diagonal entries of $\Phi_{\tau_i}^{\mu}(a_{\mathbf{t}})$ with $\tau_i\in B(\mu)$ and $i=1,2,\dots,\sharp B(\mu)$ in a row and all~$\Phi_{\tau_i}^{\mu}(a_{\mathbf{t}})$'s generate a matrix $\Phi(a_\mathbf{t})$. Similarly, we put the diagonal entries of $Q_{\tau_i}^{\mu}(a_\mathbf{t})$ in a row and all $Q_{\tau_i}^{\mu}(a_\mathbf{t})$'s generate a matrix $Q(a_\mathbf{t})$. The row of $\Phi(a_\mathbf{t})$, respectively $Q(a_\mathbf{t})$, is corresponding to $\tau_i$ with $i=1,2,\dots,\sharp B(\mu)$ and the column of $\Phi(a_\mathbf{t})$, respectively $Q(a_\mathbf{t})$, is corresponding to the weight vector in $V_\mu^K$. Explicitly,
\[\Phi^{ik}(a_\mathbf{t})=\big(\Phi_{\tau_i}^{\mu}(a_\mathbf{t})\big)^{k,k},\qquad i,k=1,2,\dots,\sharp B(\mu)\]
and
\[Q^{ik}(a_\mathbf{t})=\big(Q_{\tau_i}^{\mu}(a_\mathbf{t})\big)^{k,k},\qquad i,k=1,2,\dots,\sharp B(\mu).\]

\begin{Theorem}
Each spherical function can be written as
\[\Phi_\lambda^\mu(a_\mathbf{t})=\sum_{\nu\in B(\mu)} p_{\nu,\lambda}(\psi_1,\psi_2,\dots,\psi_n)Q_\nu^\mu(a_\mathbf{t}),\]
where $p_{\nu,\lambda}(\psi_1,\psi_2,\dots,\psi_n)$ is a polynomial in $\psi_j(\cos \mathbf{t})$'s.

Moreover, we have $\Phi(a_{\mathbf{t}})=UQ(a_{\mathbf{t}})$ where the entries of $U$ are the polynomials in $\psi_{j}(\cos \mathbf{t})$'s.
\end{Theorem}
\begin{Remark}
In this case, the diagonal of $\Phi_\lambda^\mu(a_\mathbf{t})$ can be viewed as a row vector-valued function, which can be written as $\widetilde{P}(\cos\mathbf{t})\Phi(a_\mathbf{t})$, and respectively $P(\cos\mathbf{t})Q(a_\mathbf{t})$. In this expression, $\widetilde{P}(\cos\mathbf{t})$, and respectively $P(\cos\mathbf{t})$, is a row vector-valued function and all the entries of $\widetilde{P}(\cos\mathbf{t})$ and $P(\cos\mathbf{t})$ are the polynomials in $\psi_{j}(\cos \mathbf{t})$'s.
\end{Remark}
\begin{proof}
 This theorem can be proved by Theorems \ref{thm:muaomega1bomegan} and~\ref{thm:omegasbomegan}.
\end{proof}

\begin{Remark}
In \cite{KoelvPR-JFA}, it mentions that the spherical function of $\nu+\lambda_{\rm sph}=\lambda\in P_G^+(\mu)$, where $|\lambda_{\rm sph}|\geq1$ can be written as three-recurrence relation
\begin{gather*}
\phi_{i}\Phi_{\lambda}^{\mu}=\sum_{\substack{\nu'\in B(\mu),\, \lambda_{\rm sph}'\in P_{G}^{+}(0)\\ |\lambda'|= |\lambda_{\rm sph}|-1}} d_{\nu'+\lambda'}\Phi_{\nu'+\lambda'}^{\mu}\\
\hphantom{\phi_{i}\Phi_{\lambda}^{\mu}=}{}
+\sum_{\substack{\nu'\in B(\mu),\, \lambda_{\rm sph}'\in P_{G}^{+}(0)\\ |\lambda'|= |\lambda_{\rm sph}|}} d_{\nu'+\lambda'}\Phi_{\nu'+\lambda'}^{\mu}
 +\sum_{\substack{\nu'\in B(\mu),\, \lambda_{\rm sph}'\in P_{G}^{+}(0)\\ |\lambda'|= |\lambda_{\rm sph}|+1}} d_{\nu'+\lambda'}\Phi_{\nu'+\lambda'}^{\mu}
\end{gather*}
if the following three conditions are satisfied
\begin{itemize}\itemsep=0pt
 \item[(1)] $(G,K,\mu)$ is a multiplicity free triple.
 \item[(2)] There exists $B(\mu)\subset P_{G}^{+}(\mu)$ so that for any $\lambda\in P_{G}^{+}(\mu)$, there exists unique elements $\lambda_{\rm sph}\in P_{G}^{+}(0)$ and $\nu\in B(\mu)$ such that
\[\lambda=\lambda_{\rm sph}+\nu.\]
Moreover, there is a bijection from the elements in $B(\mu)$ to the irreducible $M$-modules in~$V_{\mu}^{K}|_{M}$.
 \item[(3)] For any generator $\lambda_r$ of $P_G^+(0)$ and any $\nu\in B(\mu)$, we have that if the weight $\eta\in P(\lambda_r)$ satisfies $\nu+\eta\in P_G^+(\mu)$ and hence
\[\nu+\eta=\nu'+\lambda_{\rm sph}\]
for $\nu'\in B(\mu)$ and $\lambda_{\rm sph}\in P_G^+(0)$, then $|\lambda_{\rm sph}|\leq1$.
\end{itemize}

For the case of $\mu=a\omega_1+b\omega_n$, the three conditions are satisfied and moreover, the entries of~$U$ are all constants by Lemma \ref{lem:phiaomega1+bomegan}.

For the case of $\mu=\omega_s+b\omega_n$, the third condition is not satisfied. For instance, we define $(G,K,\mu)=({\rm SU}(9),{\rm S}({\rm U}(4)\times {\rm U}(5)),\omega_2)$. Let $\nu=\omega_1+\omega_3+\omega_7\in B(\mu)$ and $\eta=\omega_1+\omega_2-\omega_3-\omega_7+2\omega_8\in P(\omega_1+\omega_8)$, then we have
\[\nu+\eta=2\omega_1+\omega_2+2\omega_8=\omega_2+2\lambda_1.\]
So this situation contradicts the third condition.
\end{Remark}

\begin{Remark}[differential operator] For $\Phi_{\lambda}^{\mu}=PQ$, the radial part $R$ of the Casimir operator can be rewritten as
\begin{align*}
D(P)={}&-\frac{1}{2}\sum_{k=1}^{n}\frac{\partial^{2} P}{\partial t_{k}^{2}}
-\sum_{k=1}^{n}\frac{\partial P}{\partial t_{k}}\frac{\partial Q}{\partial t_{k}}Q^{-1}-(m-n)\sum_{k=1}^{n}\frac{\partial P}{\partial t_{k}}\frac{\cos t_{k}}{\sin t_{k}}\\
&\quad-\sum_{1\leq j<k\leq n} \frac{\cos(t_{j}-t_{k})\sin(t_{j}+t_{k})}
{\cos^{2}t_{k}-\cos^{2}t_{j}}\left(\frac{\partial}{\partial t_{j}}-\frac{\partial}{\partial t_{k}}\right)P\\
&\quad-\sum_{1\leq j<k\leq n} \frac{\cos(t_{j}+t_{k})\sin(t_{j}-t_{k})}
{\cos^{2}t_{k}-\cos^{2}t_{j}}\left(\frac{\partial}{\partial t_{j}}+\frac{\partial}{\partial t_{k}}\right)P
-\sum_{k=1}^{n}\frac{\cos(2t_{k})}{\sin(2t_{k})}
\frac{\partial P}{\partial t_{k}}+P L,
\end{align*}
where the entries of $L$ are polynomials in $\psi_i$'s.

Using the similar proof with \cite[Lemma 3.9]{KoelvPR-JFA}, we have
\[\sum_{j=1}^{n}\frac{\partial Q}{\partial t_{j}}\frac{\partial \psi_{i}}{\partial t_{j}}=C_{i}(\psi_{1},\psi_{2},\dots,\psi_{n})Q,\]
where all entries of $C(\psi_{1},\psi_{2},\dots,\psi_{n})$ are polynomials in $\psi_{i}$'s.
\end{Remark}

The orthogonality relation in Remark \ref{rem:ortho} can be written as
\begin{align*}
\langle\Phi_{\lambda}^{\mu},\Phi_{\lambda'}^{\mu}\rangle&=c_1\int_{A}\operatorname{Tr}(\Phi_{\lambda}^{\mu}(a_\mathbf{t})(\Phi_{\lambda'}^{\mu} (a_\mathbf{t}))^{\ast})|\delta(a_\mathbf{t})|\, {\rm d} a_\mathbf{t}\\
&=c_{1}\int_A P_{1}(\cos\mathbf{t})Q(a_\mathbf{t})(Q(a_\mathbf{t}))^\ast (P_{2}(\cos\mathbf{t}))^{\ast}|\delta(a_\mathbf{t})|\, {\rm d} a_\mathbf{t},
\end{align*}
then $Q(a_\mathbf{t})(Q(a_\mathbf{t}))^\ast$ is the matrix weight function. Let
\[S(\cos\mathbf{t})=Q(a_\mathbf{t})(Q(a_\mathbf{t}))^\ast\]
and the $(i,j)$-th entry of $S(\cos\mathbf{t})$ is
\[S^{i,j}(\cos\mathbf{t})=\underset{k=1}{\stackrel{\dim V_\mu^K}\sum}Q^{i,k}(a_\mathbf{t})\overline{Q^{j,k}(a_\mathbf{t})},\]
where $i,j=1,2,\dots,\sharp B(\mu)=\dim V_\mu^K$. Note that $S^{i,j}(\cos\mathbf{t})$ is a polynomial in $\cos^2 t_k$($k=1,2,\dots,n$). So we have
\begin{align*}
\langle\Phi_{\lambda}^{\mu},\Phi_{\lambda'}^{\mu}\rangle
&{}=c_{1}\int_{-\pi}^{\pi}\cdots\int_{-\pi}^{\pi}P_{1}(\cos\mathbf{t})S(\cos\mathbf{t}) (P_{2}(\cos\mathbf{t}))^{\ast}|\delta(a_\mathbf{t})|{\rm d}t_{1}{\rm d}t_{2}\cdots {\rm d}t_{n}\\
&{}=4^{n}c_{1}\int_{0}^{\frac{\pi}{2}}\cdots\int_{0}^{\frac{\pi}{2}}P_{1}(\cos\mathbf{t})S(\cos\mathbf{t}) (P_{2}(\cos\mathbf{t}))^{\ast}
\prod_{i=1}^n \bigl(\sin^{2(m-n)}t_{i}\sin 2t_{i}\bigr)\\
&\quad{}\times\prod_{1\leq i<j\leq n} \bigl(\sin^{2}(t_{i}+t_{j})\sin^{2}(t_{i}-t_{j})\bigr){\rm d}t_{1}{\rm d}t_{2}\cdots {\rm d}t_{n}\\
&{}=4^{n}c_{1}\int_{0}^{1}\cdots\int_{0}^{1}P_{1}(\mathbf{l})S(\mathbf{l}) (P_{2}(\mathbf{l}))^{\ast}
\prod_{i=1}^n (1-l_{i})^{m-n}\prod_{1\leq i<j\leq n} (l_{i}-l_{j})^{2}{\rm d}l_{1}{\rm d}l_{2}\cdots {\rm d}l_{n}\\
&{}=\frac{\bigl(\dim V_{\mu}^{K}\bigr)^{2}}{\dim V_{\lambda}^{G}}\delta_{\lambda\lambda'},
\end{align*}
where we let $l_i=\cos^2 t_i$ and $\mathbf{l}=(l_1,l_2,\dots,l_n)$. In this case, we let
\[c_{1}=\Big(\int_{-\pi}^{\pi}\cdots\int_{-\pi}^{\pi}|\delta(a_\mathbf{t})|\,{\rm d}t_{1}{\rm d}t_{2}\cdots {\rm d}t_{n}\Big)^{-1},\]
and $c_{1}^{-1}$ is the Selberg integration $S_{n}(\alpha,\beta,\gamma)$ with $\alpha=1$, $\beta=m-n+1$ and $\gamma=1$. Moreover, we have
\[c_1=\frac{1}{4^n}\prod_{j=0}^{n-1} \frac{\Gamma(m+1+j)}{\Gamma(1+j)\Gamma(m-n+1+j)\Gamma(j+2)}.\]

\subsection[Example: mu=omega\_1+b omega\_n]{Example: $\boldsymbol{\mu=\omega_1+b\omega_n}$}
In this case, we have $\psi_{i}^{(k)}(\cos \mathbf{t})=\psi_{i}^{(\{k\})}(\cos \mathbf{t})$ as defined in Appendix \ref{app:casimirop}, and
\[Q^{i,k}(a_\mathbf{t})=\cos t_k\cos^b t_\mathrm{N}\psi_{i}^{(k)}(\cos \mathbf{t})
 =\cos t_k\cos^b t_\mathrm{N}\sum_{\substack{|I|=i\\I\subset \mathrm{N}\backslash \{k\}}} \cos^{2}t_{I},\qquad i=0,1,\dots,n-1.\]
Then
\[S^{i,j}(\cos \mathbf{t})=\sum_{k=1}^{n}Q^{i,k}(a_\mathbf{t})\overline{Q^{j,k}(a_\mathbf{t})}=\psi_n^b\sum_{k=1}^{n}\bigl(\cos^2 t_k\psi_{i}^{(k)}(\cos\textbf{t})\psi_{j}(\cos\textbf{t})^{(k)}\bigr).\]

\begin{Lemma}
The matrix weight $S$ is indecomposable, i.e.,
\begin{gather}
\mathbb{C}I=\{A\in\operatorname{End}(\mathbb{C}^{n})\mid AS(\cos\textbf{t})=S(\cos\textbf{t})A\},\nonumber\\
\mathbb{R}I=\{A\in\operatorname{End}(\mathbb{C}^{n})\mid AS(\cos\textbf{t})=S(\cos\textbf{t})A^{\ast}\}.\label{lem:indecomposable}
\end{gather}
\end{Lemma}
\begin{proof}
Recall that the total degree of $S^{ij}(\cos\mathbf{t})$ is $2i+2j+2b+2$, the way to prove this lemma is similar to the proof of \cite[Proposition~5.1]{KoelLiu}, for which we compare the total degree of the entries between the left-hand and right-hand sides in \eqref{lem:indecomposable}.
\end{proof}

\begin{Proposition}
We have
\[S^{ij}(\cos\textbf{t})=-\psi_{n}^{b}(\cos\textbf{t})\sum_{k=-1}^{i-1} (2k+1-i-j)\psi_{k+1}(\cos\textbf{t})\psi_{i+j-k}(\cos\textbf{t})\]
for $i+j\leq n-1$ and
\[S^{ij}(\cos\textbf{t})
=-\psi_{n}^{b}(\cos\textbf{t})\sum_{k=-1}^{n-2-j} (i+j-2n+3+2k)\psi_{i+j-n+2+k}(\cos\textbf{t})\psi_{n-1-k}(\cos\textbf{t})\]
for $i+j>n-1$.
\end{Proposition}
\begin{proof}
Since $S(\cos\textbf{t})$ is a symmetric matrix, we only need to consider $S^{ij}(\cos\textbf{t})$ for which $i\leq j$. For $i=0$, by \eqref{symmetripolynomialrelation2} we have
\[S^{0j}(\cos\textbf{t})
=\psi_{n}^{b}(\cos\textbf{t})\sum_{k=1}^{n}\bigl(\cos^{2}t_{k}\psi_{j}^{(k)}(\cos\textbf{t})\bigr)
=\psi_{n}^{b}(\cos\textbf{t})(j+1)\psi_{j+1}(\cos\textbf{t})\]
and for $j=n-1$, we have, by \eqref{symmetripolynomialrelation2} and $\cos^2 t_k\psi_{n-1}^{(k)}(\cos\mathbf{t})=\psi_n(\cos\mathbf{t})$,
\[S^{i,n-1}(\cos\textbf{t})=\psi_{n}^{b}(\cos\textbf{t})\sum_{k=1}^{n}\bigl(\psi_{i}^{(k)}(\cos\textbf{t})\psi_{n}(\cos\textbf{t})\bigr)
=(n-i)\psi_{i}(\cos\textbf{t})\psi_{n}^{b+1}(\cos\textbf{t}).\]
Now we calculate other entries. We use \eqref{symmetripolynomialrelation} and \eqref{symmetripolynomialrelation2} such that
\begin{align*}
S^{ij}(\cos\textbf{t})&=\psi_{n}^{b}(\cos\textbf{t})\sum_{k=1}^{n}\bigl(\cos^{2}t_{k}\psi_{i}^{(k)}(\cos\textbf{t})\psi_{j}^{(k)}(\cos\textbf{t})\bigr)\\
&=\psi_{n}^{b}(\cos\textbf{t})\sum_{k=1}^{n}\bigl(\psi_{i+1}(\cos\textbf{t})-\psi_{i+1}^{(k)}(\cos\textbf{t})\bigr)
\bigl(\psi_{j}(\cos\textbf{t})-\cos^{2}t_{k}\psi_{j-1}^{(k)}(\cos\textbf{t})\bigr)\\
&=(i-j+1)\psi_{n}^{b}(\cos\textbf{t})\psi_{i+1}(\cos\textbf{t})\psi_{j}(\cos\textbf{t})+S^{i+1,j-1}(\cos\textbf{t}).
\end{align*}
Then we can calculate all the entries of $S$ by induction and this proposition is proved.
\end{proof}

We have
\begin{Lemma}
\[\det(S)=\psi_n^{b+1}(\cos \mathbf{t})\prod_{1\leq i<j\leq n} \bigl(\cos^2 t_j-\cos^2 t_i\bigr)^2.\]
\end{Lemma}
\begin{proof}
We only need to calculate $\det(Q)$ since $S=QQ^\ast$ and we have
\begin{align*}
 \det(Q)&=\cos^b t_\mathrm{N}
 \left|
 \begin{matrix}
 \cos t_{1}&\cos t_{2}&\cdots&\cos t_{n}\\
 \cos t_{1}\psi_{1}^{(1)}(\cos \mathbf{t})&\cos t_{2}\psi_{1}^{(2)}(\cos \mathbf{t})&\cdots& \cos t_{n}\psi_{1}^{(n)}(\cos \mathbf{t})\\
 \vdots&\vdots&\ddots&\vdots\\
 \cos t_{1}\psi_{n-1}^{(1)}(\cos \mathbf{t})&\cos t_{2}\psi_{n-1}^{(2)}(\cos \mathbf{t})& \cdots&\cos t_{n}\psi_{n-1}^{(n)}(\cos \mathbf{t})\\
 \end{matrix}
\right|\\
&=\cos^{(b+1)} t_{\mathrm{N}}\left|
 \begin{matrix}
 1&1&\cdots&1\\
 \psi_{1}^{(1)}(\cos \mathbf{t})&\psi_{1}^{(2)}(\cos \mathbf{t})&\cdots&\psi_{1}^{(n)}(\cos \mathbf{t})\\
 \vdots&\vdots&\ddots&\vdots\\
 \psi_{n-1}^{(1)}(\cos \mathbf{t})&\psi_{n-1}^{(2)}(\cos \mathbf{t})& \cdots&\psi_{n-1}^{(n)}(\cos \mathbf{t})\\
 \end{matrix}
\right|.
\end{align*}
Let $H\subset\mathrm{N}$, then by \eqref{symmetripolynomialrelation} it leads to
\begin{gather}
 \psi_{i}^{(H\cup\{k\})}(\cos \mathbf{t})-\psi_{i}^{(H\cup\{1\})}(\cos \mathbf{t})=\bigl(\psi_{i}^{(H\cup\{1,k\})}(\cos \mathbf{t})+\cos^2 t_1\psi_{i-1}^{(H\cup \{1,k\})}(\cos \mathbf{t})\bigr)\nonumber\\
 \qquad\quad{}-\bigl(\psi_{i}^{(H\cup\{1,k\})}(\cos \mathbf{t})+\cos^2 t_k\psi_{i-1}^{(H\cup \{1,k\})}(\cos \mathbf{t})\bigr)\nonumber\\
\qquad {} =\bigl(\cos^2 t_1-\cos^2 t_k\bigr)\psi_{i-1}^{(H\cup \{1,k\})}(\cos \mathbf{t}).\label{detQ}
\end{gather}
So
\begin{gather}
\det(Q)=\cos^{(b+1)} t_{\mathrm{N}}\nonumber\\
\hphantom{\det(Q)=}{} \times\left|
 \begin{matrix}
 1&0&\cdots&0\\
 \psi_{1}^{(1)}(\cos \mathbf{t})&\psi_{1}^{(2)}(\cos \mathbf{t})-\psi_{1}^{(1)}(\cos \mathbf{t})&\cdots&\psi_{1}^{(n)}(\cos \mathbf{t})-\psi_{1}^{(1)}(\cos \mathbf{t})\\
 \vdots&\vdots&\ddots&\vdots\\
 \psi_{n-1}^{(1)}(\cos \mathbf{t})&\psi_{n-1}^{(2)}(\cos \mathbf{t})-\psi_{n-1}^{(1)}(\cos \mathbf{t})& \cdots&\psi_{n-1}^{(n)}(\cos \mathbf{t})-\psi_{n-1}^{(1)}(\cos \mathbf{t})\\
 \end{matrix}
\right|\nonumber\\
\hphantom{\det(Q)}{}=\cos^{(b+1)} t_{\mathrm{N}}\prod_{i=2}^{n} \bigl(\cos^2 t_1-\cos^2 t_i\bigr)\nonumber\\
\hphantom{\det(Q)=}{} \times \left|
 \begin{matrix}
 1&1&\cdots&1\\
 \psi_{1}^{(\{1,2\})}(\cos \mathbf{t})&\psi_{1}^{(\{1,3\})}(\cos \mathbf{t})&\cdots&\psi_{1}^{(\{1,n\})}(\cos \mathbf{t})\\
 \vdots&\vdots&\ddots&\vdots\\
 \psi_{n-2}^{(\{1,2\})}(\cos \mathbf{t})&\psi_{n-2}^{(\{1,3\})}(\cos \mathbf{t})& \cdots&\psi_{n-2}^{(\{1,n\})}(\cos \mathbf{t})\\
 \end{matrix}
\right|,\label{detQfirststep}
\end{gather}
where we let $H=\varnothing$. We calculate \eqref{detQfirststep} by induction using \eqref{detQ} and
\[\det Q=\cos^{(b+1)} t_{\mathrm{N}}\prod_{1\leq i<j\leq n} \bigl(\cos^2 t_i-\cos^2 t_j\bigr).\]
So this lemma is proved.
\end{proof}

\appendix

\section[Calculation of radial part R of the Casimir operator]{Calculation of radial part $\boldsymbol{R}$ of the Casimir operator}\label{app:caloper}
The radial part $R$ of the Casimir operator is arising from an element in the universal enveloping algebra of $\mathfrak{g}$, i.e., the Casimir element. The complex result is given in Warner \cite[Proposition~9.1.2.11]{Warn-2}. The goal of this section is to give the compact type using Casselman and Mili\v{c}i\'{c}~\cite{CassM}. This appendix is a generalization of
\cite[Appendix]{KoelLiu}.
\subsection{Structure theory}
In order to calculate the radial part $R$ of the Casimir operator, we need to calculate another $K$ type which conjugates to block diagonal case. Note that $K=G^\theta$ and
\[\theta(g)=JgJ,\qquad J=\begin{pmatrix} -I_n & 0 \\ 0 & I_m \end{pmatrix}.\]
We define
\[J'=\begin{pmatrix} 0 & 0 & L_n \\ 0 & I_{m-n} & 0 \\ L_n & 0 & 0 \end{pmatrix},\qquad u=\begin{pmatrix} \frac{1}{\sqrt{2}}I_n & 0 & \frac{1}{\sqrt{2}}L_n \\ 0 & I_{m-n} & 0 \\ -\frac{1}{\sqrt{2}}L_n & 0 &\frac{1}{\sqrt{2}}I_n\end{pmatrix}.\]
Then we have $K'=G^{\theta'}=uKu^{\ast}$, where $\theta'(g)=J'gJ'$. Also $\mathfrak{k}'$, $\mathfrak{a}'$, $A'$, and $\mathfrak{m}'$ can be defined analogously. Note that $\mathfrak{m}'=u\mathfrak{m}u^\ast=\mathfrak{m}$.

Now we describe the restricted root system $\mathbf{R}$. Let
$f_i\colon \mathfrak{a}'\to \mathbb{C}$, $1\leq i \leq n$, be defined by
\[f_i \colon\ \begin{pmatrix} D & 0 & 0 \\ 0 & 0 & 0 \\ 0 & 0 &-L_{n}D L_{n}\end{pmatrix} \mapsto z_i,
\qquad D=\operatorname{diag}(z_1,\dots, z_n).\]
Then the identification of $\mathbf{R}$ is given
in Table \ref{fig:rootsystem}.

\begin{table}\renewcommand{\arraystretch}{1.2}\label{fig:rootsystem}\centering
\begin{tabular}{|l|c|l|}
\hline
$\beta\in \mathbf{R}$ & $\dim \mathfrak{g}_\beta$ & $\alpha\in \Phi$ with $\alpha\vert_{\mathfrak{a}'}=\beta$ \\
\hline
$f_i-f_j$, $1\leq i\not=j\leq n$ & $2$ & $\epsilon_i-\epsilon_j$, $\epsilon_{m+n+1-j}-\epsilon_{m+n+1-i}$ \\
\hline
$f_i+f_j$, $1\leq i<j\leq n$ & $2$ & $\epsilon_{i}-\epsilon_{m+n+1-j}$, $\epsilon_{j}-\epsilon_{m+n+n+1-i}$ \\
\hline
$2f_i$, $1\leq i\leq n$ & $1$ & $\epsilon_{i}-\epsilon_{m+n+1-i}$ \\
\hline
$f_i$, $1\leq i\leq n$ & $2(m-n)$ & \makecell[l]{$\epsilon_{i}-\epsilon_{n+j}$, $\epsilon_{n+j}-\epsilon_{m+n+1-i}$,\\ $1\leq j\leq m-n$} \\
\hline
$-f_i-f_j$, $1\leq i\neq j\leq n$ &$2$ & $\epsilon_{m+n+1-i}-\epsilon_{j}$, $\epsilon_{m+n+1-j}-\epsilon_{i}$ \\
\hline
$-2f_i$, $1\leq i\leq n$ & $1$ & $\epsilon_{m+n+1-i}-\epsilon_{i}$ \\
\hline
$-f_i$, $1\leq i\leq n$ & $2(m-n)$ & \makecell[l]{$\epsilon_{n+j}-\epsilon_{i}$, $\epsilon_{m+n+1-i}-\epsilon_{n+j}$,\\ $1\leq j\leq m-n$} \\
\hline
\end{tabular}
\caption{The restricted root system.}
\end{table}

The roots of $\Phi$ not occurring in Table \ref{fig:rootsystem} are zero when
restricted to $\mathfrak{a}'$, i.e., the roots of the form $\epsilon_{n+i}-\epsilon_{n+j}$ for
$1\leq i\not=j \leq m-n$. These roots are contained in $\mathfrak{m}$. Also we have
\[\mathfrak{g} = \mathfrak{a}' \oplus \mathfrak{m} \oplus \bigoplus_{\beta\in \mathbf{R}} \mathfrak{g}_\beta.\]
We define
\[\Delta^{+}=\{\alpha\in\Phi^{+}\mid\alpha|_{\mathfrak{a}'}\neq0\},\qquad\Delta^{-}
=\{\alpha\in\Phi^{-}\mid\alpha|_{\mathfrak{a}'}\neq0\},\qquad\Delta=\Delta^{+}\cup\Delta^{-}.\]
We define $\{ f_i-f_j\mid 1\leq i,j\leq n,\, i\neq j\}\cup\{f_i+f_j \mid 1\leq i,j\leq n,\, i\neq j\}$ as the middle roots, $\{\pm2f_i\}_{i=1}^n$ as the long roots, and $\{\pm f_i\}_{i=1}^n$ as the short roots.

\begin{Remark}
For $m>n$, the restricted root system is of $BC_{n}$ type. For $m=n$, the restricted root system is of $C_{n}$ type.
Also for $m=n$, all the matrices in this section can be written as $2\times 2$ block matrices. Then the restricted root system only includes middle roots and long roots since the dimension of the short root space is $2(m-n)=0$, for $m=n$.
\end{Remark}

In order to calculate the Weyl group of the restricted root system, we need to calculate
$W = N_{K'}(A')/Z_{K'}(A')$. Note that $M=Z_{K'}(A')=Z_K(A)$.

\begin{Lemma}\label{lem:NKa} We have
\begin{gather*}
N_{K'}(A') = \Bigg\{ k= \begin{pmatrix} a & 0 & c L \\ 0 & e & 0 \\ L c & 0 & L a L \end{pmatrix} \in K'
\mid a+c \in \mathcal{P}_n, \\
\hphantom{N_{K'}(A') = \Bigg\{}{}\forall r\in\{1,\dots, n\} \, \exists!
s\in \{1,\dots, n\}\cup \{ m+1,\dots, m+n\}\ \, k_{r,s}\not=0 \Bigg\}.
\end{gather*}
Let $P_i$ be the $n\times n$ permutation matrix corresponding to the
transposition $(i,i+1)$, put
\begin{gather*}
s_i = \begin{pmatrix} P_i & 0 & 0 \\ 0 & I_{m-n} & 0 \\ 0 & 0 & LP_iL \end{pmatrix}, \quad 1\leq i \leq n-1,
\qquad
s_n = \begin{pmatrix} x & 0 & yL \\ 0 & I_{m-n} & 0 \\ Ly & 0 & LxL\end{pmatrix},
\end{gather*}
where $x=\operatorname{diag}(1,\dots, 1,0)$, $y=\operatorname{diag}(0,\dots, 0,1)$ are $(n\times n)$-matrices.
Then the elements $s_k$, $k\in \{1,\dots,n\}$ satisfy the quadratic and the braid
relations
\begin{gather*}
s_i^2 =1,\qquad s_is_j = s_js_i,\quad |i-j|>1, \\
s_{i+1}s_i s_{i+1} = s_i s_{i+1}s_i, \quad 1\leq i < n-1, \qquad
s_{n-1}s_n s_{n-1}s_n = s_ns_{n-1}s_n s_{n-1}.
\end{gather*}
\end{Lemma}
\begin{proof}
It can be proved by brute force.
\end{proof}

Note that the Weyl group for the restricted root system corresponds to
the hyperoctahedral group, i.e., the wreath product
$S_n \wr \mathbb{Z}_2 = S_n \ltimes \mathbb{Z}^n_2$. The Dynkin diagram is given in Figure \ref{fig:Dynkindiagram}.

\begin{figure}[h]\centering
\begin{tikzpicture}[scale=0.7]
% \draw [help lines, thin] (0,0) grid (18,9);
\draw [ultra thick] (6,4) circle [radius=0.2];
\draw [ultra thick] (8,4) circle [radius=0.2];
\draw [ultra thick] (11,4) circle [radius=0.2];
\draw [ultra thick] (13,4) circle [radius=0.2];
\draw [ultra thick] (15,4) circle [radius=0.2];
\draw [ultra thick] (17,4) circle [radius=0.2];
\draw[very thick] (6.2,4) to [out=0,in=180] (7.8,4);
\draw[dashed, very thick] (8.2,4) to [out=0,in=180] (10.8,4);
\draw[very thick] (11.2,4) to [out=0,in=180] (12.8,4);
\draw[very thick] (13.2,4) to [out=0,in=180] (14.8,4);
\draw[very thick] (15,4.2) to [out=0,in=180] (17,4.2);
\draw[very thick] (15,3.8) to [out=0,in=180] (17,3.8);
\end{tikzpicture}
\caption{Dynkin diagram with $n$ nodes for the restricted Weyl group.}\label{fig:Dynkindiagram}
\end{figure}
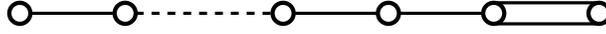

\subsection[Radial part R of the Casimir operator]{Radial part $\boldsymbol{R}$ of the Casimir operator}\label{app:rofcasimiroperator}

We recall the Casimir element in the universal enveloping algebra $\mathfrak{U}(\mathfrak{g})$ and it can be written as
\[\Omega=\sum X_{i}\widetilde{X}_{i},\]
where $X_{i}$ is the basis of $\mathfrak{g}$ and $\widetilde{X}_{i}$ is the corresponding dual basis. We define the dual basis by the Killing form $B(X,Y)=\operatorname{Tr}(XY)$.

We use the orthogonal decomposition
\[\mathfrak{g}=\mathfrak{m}\oplus\mathfrak{a}'\oplus\mathfrak{n}.\]
In this equation, $\mathfrak{n}$ is spanned by the root space vector corresponding to the roots in $\Delta$.

Note that $B|_{\mathfrak{m}\times\mathfrak{m}}$ and $B|_{\mathfrak{a}'\times\mathfrak{a}'}$ are non-degenerate. Moreover, denoting the Casimir element of $\mathfrak{m}$, respectively $\mathfrak{a}'$, by $\Omega_{\mathfrak{m}}$, respectively $\Omega_{\mathfrak{a}'}$, we have
\[\Omega=\Omega_{\mathfrak{m}}+\Omega_{\mathfrak{a}'}
+\sum_{\alpha\in\Delta^{+}} (Y_{\alpha}Y_{-\alpha}+Y_{-\alpha}Y_{\alpha})\]
and for $\epsilon_{i}-\epsilon_{j}\in\Delta$, we define $Y_{\epsilon_{i}-\epsilon_{j}}=E_{ij}$ which spans the root space $\mathfrak{g}_{\epsilon_{i}-\epsilon_{j}}$.

For $\alpha\in\Delta$, we define
\[X_{\alpha}=Y_{\alpha}+Y_{\theta\alpha}.\]
We have
\begin{align*}
\Omega=\Omega_{\mathfrak{m}}+\frac{1}{2}\sum_{i=1}^n H_{ii}H_{ii}
&-\sum_{\alpha\in\Delta^{+}} \frac{1}{(\alpha(a)-\alpha(a)^{-1})^{2}}
(2(X_{\alpha}^{a}X_{-\alpha}^{a}+X_{\alpha}X_{-\alpha})\\
&-\bigl(\alpha(a)+\alpha(a)^{-1}\bigr)(2X_{\alpha}^{a}X_{-\alpha}-(\alpha(a)-\alpha(a)^{-1})H_{\alpha|_{\mathfrak{a}'}}))
\end{align*}
by calculation.

We put
\[\mathfrak{A}=\mathfrak{U}(\mathfrak{a}')\otimes\mathfrak{U}(\mathfrak{k}')
{\otimes}_{\mathfrak{U}(\mathfrak{m})}\mathfrak{U}(\mathfrak{k}'),\]
then we have
\begin{Lemma}[{\cite[Theorem 2.4]{CassM}}]\label{casielementtocasioperator}
Let $a\in (A')^\mathbb{C}$ such that $\alpha(a)\neq\pm1$ for $\alpha\in\Delta$, and $F\colon (A')^{\mathbb{C}}\to\operatorname{End}\big(V_\mu^K\big)$, then
\[\Gamma_{a}\colon \ \mathfrak{A}\rightarrow \mathfrak{U}(\mathfrak{g}),\]
where
\[\Gamma_{a}(H\otimes X\otimes Y)=X^{a}HY\]
is a linear isomorphism. Moreover, $H\otimes X\otimes Y$ acts on the function $F(a)$ by
\[(H\otimes X\otimes Y)F(a)=\pi_\mu^K(X)\frac{{\rm d}}{{\rm d}z}\Big|_{z=0}F(a\exp(zH))\pi_\mu^K(Y).\]
\end{Lemma}

\begin{Remark}\label{rem:campactOmega}
We restrict $a\in (A')^\mathbb{C}$ to the compact type such that
\[a=\operatorname{diag}\bigl({\rm e}^{\mathrm{i}t_1},\dots, {\rm e}^{\mathrm{i}t_{n}}, \underbrace{1,\dots, 1}_{m-n}, {\rm e}^{-\mathrm{i}t_n},\dots, {\rm e}^{-\mathrm{i}t_{1}}\bigr),\qquad t_i\in\mathbb{R},\]
then we have
\[(H\otimes X\otimes Y)F(a)=\mathrm{i}\pi_\mu^K(X)\frac{{\rm d}}{{\rm d}t}\Big|_{t=0}F(a\exp(tH))\pi_\mu^K(Y).\]
\end{Remark}

We separate the radial part $R$ of the Casimir operator, i.e., $\Gamma_{a}^{-1}(\Omega)$, into five parts which are $M$-scalar part, second order differential operator part, short root part, middle root part, and long root part. We have
\begin{align*}
(RF)(a_{\mathbf{t}})&=\bigl(\Gamma_{a}^{-1}\Omega\bigr)(F)(a_{\mathbf{t}})\\
&=\pi_\mu^K(\Omega_{\mathfrak{m}})F(a_{\mathbf{t}})
-\frac{1}{2}\sum_{j=1}^{n}\frac{\partial^{2}}{\partial t_{j}^{2}}F(a_{\mathbf{t}})+(R_{s}F)(a_{\mathbf{t}})+(R_{m}F)(a_{\mathbf{t}})+(R_{l}F)(a_{\mathbf{t}}).
\end{align*}
In this operator, the short root part is
\begin{align*}
(R_{s}F)(a_{\mathbf{t}})={}&\sum_{1\leq j\leq n} \sum_{\substack{\alpha\in\Delta^{+}\\ \alpha|_{A}=\beta_{j}}} \frac{1}{2\sin^{2} t_{j}}
\biggl(\bigl(\pi_\mu^K(X_{\alpha})\pi_\mu^K(X_{-\alpha})F(a_{\mathbf{t}})+F(a_{\mathbf{t}})\pi_\mu^K(X_{\alpha})\pi_\mu^K(X_{-\alpha})\bigr)
\\
&{}-\cos t_{j}\biggl(2\pi_\mu^K(X_{\alpha})F(a_{\mathbf{t}})\pi_\mu^K(X_{-\alpha})+\sin t_{j}\frac{\partial}{\partial t_{j}}F(a_{\mathbf{t}})\biggr)\biggr),
\end{align*}
the middle root part is
\begin{align*}
(R_{m}F)(a_{\mathbf{t}})={}&\sum_{1\leq j<k \leq n} \sum_{\substack{\alpha\in\Delta^{+}\\ \alpha|_{A}=\beta_{j}-\beta_{k}}} \frac{1}{2\sin^{2}(t_{j}-t_{k})}
\bigg(\bigl(\pi_\mu^K(X_{\alpha})\pi_\mu^K(X_{-\alpha})F(a_{\mathfrak{t}})\\
&{}+F(a_{\mathfrak{t}})\pi_\mu^K(X_{\alpha})\pi_\mu^K(X_{-\alpha})\bigr)-\cos(t_{j}-t_{k})\bigg(2\pi_\mu^K(X_{\alpha})F(a_{\mathfrak{t}})\pi_\mu^K(X_{-\alpha})
\\
&{}+\sin(t_{j}-t_{k})\bigg(\frac{\partial}{\partial t_{j}}-\frac{\partial}{\partial t_{k}}\bigg)F(a_{\mathfrak{t}})\bigg)\bigg)\\
&{}+\sum_{1\leq j<k \leq n} \sum_{\substack{\alpha\in\Delta^{+}\\ \alpha|_{A}=\beta_{j}+\beta_{k}}} \frac{1}{2\sin^{2}(t_{j}+t_{k})}
\bigg(\big(\pi_\mu^K(X_{\alpha})\pi_\mu^K(X_{-\alpha})F(a_{\mathfrak{t}})\\
&{}+F(a_{\mathfrak{t}})\pi_\mu^K(X_{\alpha})\pi_\mu^K(X_{-\alpha})\big)
-\cos(t_{j}+t_{k})\bigg(2\pi_\mu^K(X_{\alpha})F(a_{\mathfrak{t}})\pi_\mu^K(X_{-\alpha})
\\
&{}+\sin(t_{j}+t_{k})\bigg(\frac{\partial}{\partial t_{j}}+\frac{\partial}{\partial t_{k}}\bigg)F(a_{\mathfrak{t}})\bigg)\bigg),
\end{align*}
and the long root part is
\begin{align*}
(R_{l}F)(a_{\mathbf{t}})={}&\sum_{1\leq j\leq n} \sum_{\substack{\alpha\in\Delta^{+}\\ \alpha|_{A}=2\beta_{j}}} \frac{1}{2\sin^{2}(2t_{j})}
 \bigg(\big(\pi_\mu^K(X_{\alpha})\pi_\mu^K(X_{-\alpha})F(a_{\mathfrak{t}})+F(a_{\mathfrak{t}})\pi_\mu^K(X_{\alpha})\pi_\mu^K(X_{-\alpha})\big)\\
&{}-\cos(2t_{j})\bigg(2\pi_\mu^K(X_{\alpha})F(a_{\mathfrak{t}})\pi_\mu^K(X_{-\alpha})+\sin(2t_{j})
\frac{2\partial}{\partial t_{j}}F(a_{\mathfrak{t}})\bigg)\bigg).
\end{align*}
\begin{Remark}
For $\mu=\omega_s+b\omega_n$, we have
\[\pi_\mu^K(\Omega_{\mathfrak{m}}) F(a_{\mathbf{t}})
=-\frac{2 {{s}^{2}}+(( 2 b-1) n+( -2 b-1) m) s+{{b}^{2}}{{n}^{2}}-{{b}^{2}} m n}{2 n+2 m}F(a_\mathbf{t}).\]
\end{Remark}

\subsection{Special cases}\label{app:casimirop}
The proof of Lemma \ref{lem:rqi} is a lengthy, but explicit calculation, which has been done with the help of computer algebra, in particular \textsc{Maxima}. We introduce some preliminaries in this section for proving Lemma \ref{lem:rqi}.

For $H\subset\{1,2,\dots,n\}$ with $|H|=s$, we define
\[\psi_i^{(H)}(\cos\mathbf{t})=\sum_{\substack{|I|=i\\I\subset \mathrm{N}\backslash H}} \cos^{2}t_{I},\qquad i=0,1,\dots,n-s,\]
then $\psi_i^{\varnothing}=\psi_i$. In this case, we rewrite $\psi_i^{(k)}$ as $\psi_i^{(\{k\})}$.

We have
\begin{equation}\label{symmetripolynomialrelation}
\psi_{i}^{(H\cup \{j\})}(\cos\textbf{t})+\cos^{2}t_{j}\psi_{i-1}^{(H\cup \{j\})}(\cos\textbf{t})=\psi_{i}^{(H)}(\cos\textbf{t}).
\end{equation}
It leads to
\begin{gather}
\sum_{j\in \mathrm{N}\backslash H} \psi_{i}^{(H\cup\{j\})}(\cos\textbf{t})=(n-s-i)\psi^{(H)}_{i}(\cos\textbf{t}),\nonumber\\ \sum_{j\in \mathrm{N}\backslash H} \cos^{2}t_{j}\psi_{i-1}^{(H\cup\{j\})}(\cos\textbf{t})=i\psi_{i}^{(H)}(\cos\textbf{t}).\label{symmetripolynomialrelation2}
\end{gather}

Note that \eqref{symmetripolynomialrelation} follows by a direct calculation or by use of the generating function, see \cite[Section~I.2]{Mac}, for the elementary symmetric functions. By \eqref{symmetripolynomialrelation} and differentiating the generating function, we get \eqref{symmetripolynomialrelation2}.

By using \eqref{symmetripolynomialrelation} twice, we have
\begin{align*}%\label{psiis3parts}
\psi_{i}^{(H)}(\cos\mathbf{t})={}&\cos^{2}t_{j}\cos^{2}t_{k}\sum_{\substack{|I|=i-2\\I\subset \mathrm{N}\backslash (H\cup\{j,k\})}} \cos^{2}t_{I}+(\cos^{2}t_{j}+\cos^{2}t_{k})\sum_{\substack{|I|=i-1\\I\subset \mathrm{N}\backslash (H\cup\{j,k\})}} \cos^{2}t_{I}
\\&{}+\sum_{\substack{|I|=i\\I\subset \mathrm{N}\backslash (H\cup\{j,k\})}} \cos^{2}t_{I}\nonumber\\
={}&\cos^{2}t_{j}\cos^{2}t_{k}\psi_{i-2}^{(H\cup\{j,k\})}(\cos\mathbf{t})\\
&{}+(\cos^{2}t_{j}+\cos^{2}t_{k})\psi_{i-1}^{(H\cup\{j,k\})}(\cos\mathbf{t})+\psi_{i}^{(H\cup\{j,k\})}(\cos\mathbf{t}).
\end{align*}

\subsection*{Acknowledgements}

We thank Erik Koelink and Maarten van Pruijssen for useful discussions and feedback, and we also thank the anonymous referees for contributing to improve the paper.

%\bibliographystyle{sigma}
%\bibliography{sigma23-04Liu}

\pdfbookmark[1]{References}{ref}
\LastPageEnding

\end{document}